\documentclass[times,11pt,reqno,a4paper,oneside]{amsart}
\usepackage{enumitem}%

\usepackage{color}
\usepackage[dvipsnames]{xcolor}
\usepackage[pagebackref,colorlinks]{hyperref}
\hypersetup{colorlinks,breaklinks,
            urlcolor={blue!80!black},
            linkcolor={black!40!blue},
            citecolor={blue!50!red}
}
\usepackage[margin=1.25in]{geometry}

\usepackage{tikz}
\usepackage{tikz-cd}
\usetikzlibrary{positioning,cd}
\usetikzlibrary{shapes,arrows,decorations,calligraphy,fit}

\usepackage{amsmath,amssymb,mathrsfs}

\usepackage{url}
\usepackage{hyperref}
\hypersetup{
	colorlinks=true, linktocpage=true, hidelinks
}

\usepackage{etoolbox}
\apptocmd{\sloppy}{\hbadness 10000\relax}{}{}

\newtheorem{theorem}{\rm\bf Theorem}[section]
\newtheorem{proposition}[theorem]{\rm\bf Proposition}
\newtheorem{lemma}[theorem]{\rm\bf Lemma}
\newtheorem{corollary}[theorem]{\rm\bf Corollary}

\newtheorem*{theorem*}{Theorem}
\newtheorem*{theorem 1}{\rm\bf Proposition 1}
\newtheorem*{theorem 2}{\rm\bf Proposition 2}

\theoremstyle{definition}
\newtheorem{definition}[theorem]{\rm\bf Definition}

\theoremstyle{remark}
\newtheorem{remark}[theorem]{\rm\bf Remark}

\newtheorem{example}[theorem]{\rm\bf Example}

\def\half#1#2{\begin{matrix}\frac{#1}{#2}\end{matrix}}

\def\R#1{\mathbb{R}^{#1}}

\def\Hurw{\mathbb{F}}

\def\x{\overline{x}}

\def\scal#1#2{\langle #1, #2 \rangle}
\def\SCAL#1#2{h( #1, #2)    }

\def\Alg{\mathbb{A}}

\def\xp{x_0}
\def\xm{x_1}

\DeclareMathOperator{\im}{im}

\DeclareMathOperator{\trace}{tr}

\DeclareMathOperator{\End}{End}

\newcommand{\sphere}{\mathbb{S}}

\newcommand{\quat}{\mathbb{H}}
\newcommand{\cayley}{\mathbb{O}}

\newcommand{\fie}{\mathbb{F}}

\newcommand{\so}{\mathfrak{so}}

\newcommand{\trip}{{T}}

\newcommand{\cmlt}{\times}

\newcommand{\cone}{\mathsf{C}}

\newcommand{\alg}{\mathbb{A}}

\newcommand{\om}{\omega}

\newcommand{\mlt}{\circ}

\newcommand{\hess}{\operatorname{Hess}}

\renewcommand{\part}{\vdash}

\newcommand{\Id}{\operatorname{Id}}

\newcommand{\id}{\operatorname{Id}}

\newcommand{\lap}{\Delta}

\newcommand{\la}{\lambda}

\newcommand{\eno}{\operatorname{End}}

\newcommand{\si}{\sigma}

\newcommand{\re}{\operatorname{Re}}

\newcommand{\lb}{\langle}
\newcommand{\ra}{\rangle}

\newcommand{\gl}{\mathfrak{gl}}

\newcommand{\g}{\mathfrak{g}}

\newcommand{\rea}{\mathbb R}
\newcommand{\com}{\mathbb C}

\newcommand{\tr}{\operatorname{tr}}

\newcommand{\minop}{\mathscr{M}}

\def\vt{{v}}
\def\xt{{x}}
\def\wt{{w}}

\begin{document}

\title{Algebraic constructions of cubic minimal cones}

\author[D.~J.~F.~Fox]{Daniel J. F. Fox}
\address[D.~J.~F.~Fox]{Departamento de Matemática Aplicada\\ Escuela Técnica Superior de Arquitectura\\ Universidad Politécnica de Madrid\\ Av. Juan de Herrera 4\\ 28040 Madrid\\ Spain}
\email{daniel.fox@upm.es}

\author[V.G.~Tkachev]{Vladimir G. Tkachev}
\address[V.G.~Tkachev]{Department of Mathematics\\ Link\"oping University\\ 58183\\ Sweden}
\email{vladimir.tkatjev@liu.se}

\dedicatory{Dedicated to Nikolai Nadirashvili on the occasion of his 70th birthday}

\begin{abstract}
Hsiang algebras are a class of nonassociative algebra defined in terms of a relation quartic in elements of the algebra.
This class arises naturally in relation to the construction of real algebraic minimal cones. Additionally, Hsiang algebras were crucial in the construction of singular (trulsy viscosity) solutions of nonlinear uniformly elliptic partial differential equations in a series of papers by Nikolai Nadirashvili and Serge Vl\u{a}du\c{t}. The classification of Hsiang algebras is a challenging problems, based on a Peirce decomposition like that used to study Jordan algebras, although more complicated.
This paper introduces two new tools for studying Hsiang algebras: a distinguished class of algebras called \textit{quasicomposition} that generalize the Hurwitz algebras (the reals, complexes, quaternions, and octonions) and cross-product algebras and a tripling construction associating with a given algebra one of three times the dimension of the original algebra, that can be thought of as a kind of analogue of the Cayley-Dickson doubling process. One main result states that the triple is an exceptional Hsiang algebra if and only if the original algebra is quasicomposition. The quasicomposition and tripling notions are interesting in their own right and will be considered in a more general form elsewhere.
\end{abstract}



\subjclass[2000]{
}

\maketitle

\setcounter{tocdepth}{1}  
\tableofcontents

%
 \bigskip




%

\section{Introduction}
In this paper, we discuss why the construction and classification of certain commutative nonassociative algebras called Hsiang and quasicomposition algebras are relevant and interesting from an analytic and differential geometric point of view. Nonassociative algebraic structures occur naturally in relation to the description of homogeneous and special model solutions in a wide variety of geometric and analytic problems.
A typical example is the relation between finite-dimensional Jordan algebras and the geometry of self-dual homogeneous cones and the associated tube domains and the analysis of functions on them \cite{FKbook, Koecher}.
The approach to nonassociative commutative algebras taken here is related technically to that employed in recently studied contexts that include the Majorana, axial, and code algebras which axiomatize properties of the Griess algebras associated with certain sporadic finite simple groups and vertex operator algebras \cite{DeMedts-Peacock-Shpectorov-VanCouwenberghe, Griess85, Ivanov09, Ivanov15, HSS18b, Rehren17},
the construction of viscosity solutions to nondivergent elliptic type PDEs \cite{NV2007, NV2011a,NTV}, and commutative Killing-Einstein and Hessian algebras \cite{Fox2021a,Fox2021b,Fox2022}.

The main motivation for the present paper is the problem first posed by Wu-yi Hsiang of classifying real algebraic minimal (meaning mean curvature zero) cones \cite{Hsiang67}.
The understanding of the algebraic structure and variational properties of minimal cones in $\R{n}$ is a central issue in global analysis and geometry \cite{Yau92, Yau2000}.
Quadratic minimal cones (the Clifford-Simons minimal cones) play a basic role in the resolution of the Bernstein problem for minimal hypersurfaces in $\R{n}$ for $n\le7$ \cite{Almgren, Fleming, Simons} and the construction of counterexamples for $n=8$ \cite{BGG}.

Recent developments in the structure of singular (truly viscosity) solutions of nonlinear  elliptic partial differential equations initiated by N.~Nadirashvili and S.~Vl{\u{a}}du{\c{t}} in \cite{NV2007}, and later in the context of \emph{cubic} (degree 3) minimal cones and radial eigencubics have shown remarkable connections with nonassociative algebras \cite{NTVbook, Tk19a}.
More precisely,  for certain \emph{exceptional radial eigencubics} $u(x)$, the function $u(x)/|x|$ is a viscosity solution of a Hessian uniformly elliptic equation and of a uniformly elliptic Isaacs equation in the unit ball \cite{NTVbook}.
Although general algebraic minimal cones (of degree 3 and higher) have been the subject of considerable recent interest, still little is known about their structure.
Even the construction of particular examples remains difficult.
On the other hand, known examples of cubic minimal cones arise in diverse algebraic, analytical, and geometrical contexts. We refer the interested reader to \cite{Tk19a, Tk19M} and the references therein for further information.


A real algebraic minimal cone is determined by a homogeneous degree $k$ polynomial solution of the equation
\begin{align}\label{pre1laplace}
|D u|^2\Delta u-\half{1}{2}\scal{Du}{D|Du|^2}=B(x)u(x), \qquad x\in \R{n}.
\end{align}
for some degree $2k - 4$ homogeneous polynomial $B$ on the righthand side (see \eqref{hsianganalytic} below). Here, the cubic case of \eqref{pre1laplace} is called the \emph{nonradial Hsiang equation}. Finding degree $3$ real algebraic minimal hypersurfaces in spheres is equivalent to solving the Euclidean nonradial Hsiang algebras over the real field. For the cubic case $k = 3$, $B$ must be a quadratic form. In \cite[p.~258, 265]{Hsiang67}, Wu-yi~Hsiang proposed classifying homogeneous cubic polynomial solutions on $\R{n}$ in the special case of \eqref{pre1laplace} where $B$ is a multiple of the \textit{Euclidean quadratic form} $h(x,x)$:
\begin{align}\label{1laplace}
|D u|^2\Delta u-\half{1}{2}\scal{Du}{D|Du|^2}=-\frac32\theta h(x,x)u(x), \qquad x\in \R{n}.
\end{align}
By rescaling $h$, the scalar factor $\theta$ can be normalized as one likes; the constant factor $-3/2$ is convenient later.
Solutions of \eqref{1laplace} were called \textit{radial eigencubics} in \cite{NTVbook}.
All known \textit{irreducible} examples of cubic solutions of the nonradial Hsiang equations are in fact radial eigencubics, and here we focus on the radial case.

By using invariant theory, Hsiang constructed four explicit examples of radial eigencubics for dimensions $n \in \{5,8,9, 15\}$ and posed the problem of classifying radial eigencubics.
In \cite[Chapter $6$]{NTVbook} steps were taken toward the solution of this problem. This was achieved by translating the problem into one of pure algebra, the classification of Euclidean radial Hsiang algebras. This classification problem is viable because for such algebras there is a Peirce type decomposition associated with an idempotent, like that use in the classification of Jordan algebras, although somewhat more complicated. As is recounted below, there is an infinite family of radial eigencubics, occurring in arbitrarily large dimensions, that are constructed from Clifford systems, and there are finitely many dimensions in which there occur radial eigencubics, here denominated \emph{exceptional}, not equivalent to ones obtained from Clifford systems. For the Hsiang algebras associated with exceptional radial eigencubics the trace-form, called here the \emph{Killing form}, constructed by tracing the product of two left multiplication endomorphisms, is invariant. This Killing invariance condition almost distinguishes the exceptional radial Hsiang algebras; only four of the Hsiang algebras built from Clifford systems, those here called \emph{mutants}, possess this property. The classification remains incomplete because the exceptional Hsiang algebras are not fully understood. Although they occur in finitely many dimensions no greater than $72$, existence and uniqueness are not fully determined for some of the possible dimensions.
An important feature of the exceptional eigencubics is that they appear in a variety of geometric and analytic problemas.
This makes their characterization an attractive and interesting problem. One of the goals of the present paper is to provide an alternative approach to understanding of the algebraic nature of radial eigencubics, and, especially, exceptional eigencubics.

A phenomenological observation is that all of the exceptional radial eigencubics occur in a dimension divisible by $3$ except for five of them constructed from the cubic isoparametric polynomials of Cartan and associated with the algebras here called \emph{eikonal} that were characterized by the second author in \cite{Tk14}.
This observation suggested the relevance of the triple algebra construction described in section \ref{sec:trip}. The triple algebra construction associates with a metrized commutative algebra of dimension $n$ a commutative algebra of dimension $3n$. It extends the Nahm algebra of a Lie algebra constructed by Kinyon and Sagle \cite{kinyon2002nahm} and a similar construction for commutative algebras described by the first author in \cite{Fox2009, Fox2015, Fox2021a}. It can be though of as something like the Cayley-Dickson process used to construct the octonions from the quaternions. From an analytic point of view it can be thought of as a systematic way of adding variables to solve a differential equation.

Theorem \ref{th:tuda} shows that the triple of an algebra is radial Hsiang (necessarily exceptional of dimension divisible by $3$) if and only if the original algebra is what here is called \emph{quasicomposition}. The quasicomposition algebras are a class of algebras that includes the well understood Hurwitz algebras (reals, complexes, quaternions, and octonions), the $3$ and $7$ dimensional cross-product algebras, the $6$-dimensional vector color algebra, some twisted forms of these, and conjecturally not much more. This class of algebras is interesting in its own right because it groups together these closely related and widely occurring algebras as examples of a single structure with a common definition.

\smallskip
Diagram \ref{diagram} groups the Hsiang algebras by their principal characteristic propeties and the comments below summarize the analytic and geometric structures and properties distinguishing these algebras.

%

\renewcommand{\figurename}{Diagram}
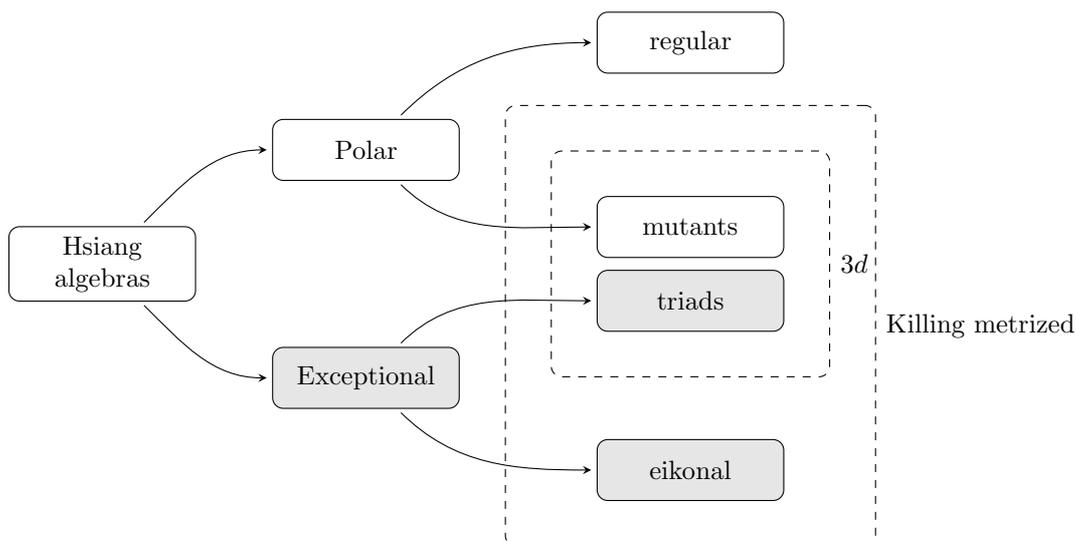
\begin{figure}[h]
\centering
\begin{tikzpicture}[>=stealth,font=\small,
force/.style={shape=rectangle,draw,rounded corners,text width=0.15\textwidth, minimum height=2.1em,text badly centered},
cloud/.style={shape=ellipse,draw,rounded corners,text width=0.15\textwidth, minimum height=2.1em,text badly centered},
every fit/.style={rectangle,rounded corners,draw,dashed,inner sep=1.55em}]

\node [force](Hs) {Hsiang algebras};
\node [force, above right =0.6 and 1 of Hs](polar) {Polar};
\node [force,fill={black!10!white},below right =0.6 and 1  of Hs] (exc) {Exceptional};

\node [force,above right =0.6 and 1.8 of polar] (reg) {regular};
\node [force,below right =0.2 and 1.8 of polar] (mut) {mutants};
\node [force,fill={black!10!white},above right =0.2 and 1.8 of exc] (eik1) {triads};
\node [force,fill={black!10!white},below right =0.4 and 1.8 of exc] (eik2) {eikonal};

\draw[->, shorten >=2pt, shorten <=2pt] (Hs) to[out=45,in=180] (polar);
\draw[->, shorten >=2pt, shorten <=2pt] (Hs) to[out=315,in=180] (exc);
\draw[->, shorten >=2pt, shorten <=2pt] (polar) to[out=45,in=180] (reg);
\draw[->, shorten >=2pt, shorten <=2pt] (polar) to[out=315,in=180] (mut);
\draw[->, shorten >=2pt, shorten <=2pt] (exc) to[out=45,in=180] (eik1);
\draw[->, shorten >=2pt, shorten <=2pt] (exc) to[out=315,in=180] (eik2);


\node [fit=(eik1) (mut),label=right:$3d$] (S) {};
\node [fit=(S) (eik2), label=right:Killing metrized] {};

\end{tikzpicture}
\label{fig1}
\caption{Radial Hsiang algebras organized by their distinguishing properties}\label{diagram}
\end{figure}

The following list summarizes the key features indicated in diagram \ref{diagram}. These are all explained in the main text.
\begin{itemize}
\item
\textbf{Polar}: constructed from symmetric Clifford systems, they carry a $\mathbb{Z}_2$-grading, have large singular locus, and there are infinitely many occurring in arbitrarily high dimensions.
\begin{itemize}
\item[$\diamond$]
\emph{Regular}: distinguished by uniqueness of the polar decomposition, they are not Killing metrized.
\item[$\diamond$]
\emph{Mutants}: only four examples built directly by tripling Hurwitz algebras, they admit several polar decompositions permuted by a canonical $S_{3}$ action and are Killing metrized.
\end{itemize}
\item
\textbf{Exceptional}: occur in only finitely many dimensions, they have a hidden simple Jordan algebra structure, and are Killing metrized.
\begin{itemize}
\item[$\diamond$]
\emph{Eikonal}: only five examples associated with cubic isoparametric polynomials, they are only examples having trivial singular locus (smooth).
\item[$\diamond$]
\emph{Triads}: dimension divisible by three, triality-like properties, nontrivial singular locus. These are the least understood examples.
\end{itemize}
\end{itemize}

\medskip

The paper is structured as follows. Section~\ref{sec:prelim} describes necessary background on algebraic minimal cones and the original Hsiang problem.
It also presents some particular solutions (radial eigencubics) that are revisited later in later sections. The general notion of a metrized algebra with involution is recalled in Section~\ref{sec:metr}, and the particular case of commutative metrized algebras and their description in terms of analytic properties of an associated cubic form are considered. With these tools in hand, in Section~\ref{sec:Hsiang} we reformulate the Hsiang problem in the metrized algebra language and recall some known results about polar and exceptional Hsiang algebras and their Peirce decompositions that are needed in subsequent sections. Theorem \ref{th:trace} characterizes Killing-metrized Hsiang algebras. Section~\ref{sec:Hurw} recalls the definitions of composition and cross-product algebras and introduces the notion of quasicomposition algebras and studies their basic properties, including a basic invariant called the defect. The tripling construction is described in Section~\ref{sec:trip} and there is proved the result relating quasicomposition algebras with a certain subclass of Hsiang algebras obtained via the tripling construction. It turns out that the Hsiang algebras obtained via tripling are always exceptional or near exceptional (of Hurwitz type). In Section~\ref{sec:Delta} their structural (Peirce) dimensions are identified and expressed in terms of the defects of the associated quasicomposition algebras.

\section{Hsiang eigencubics}\label{sec:prelim}

\subsection{Minimal algebraic cones and Hsiang equations}


A smoothly immersed submanifold of a Riemannian manifold $M$ is said to be \textit{minimal} if its mean curvature is zero (in particular, the qualifier \emph{minimal} supposes no further variational restrictions on $M$).

Let $\sphere^{n-1}$ be the unit sphere in Euclidean space $\rea^{n}$.
To $M \subset \sphere^{n-1}$ associate the cone $\cone(M) = \{tp: (p, t) \in M \times [0, \infty)\} \subset \rea^{n}$. A closed smoothly immersed submanifold $M \subset\sphere^{n-1}$ is minimal if and only if $\cone(M) \setminus \{0\}$ is an immersed minimal submanifold in $\rea^{n}$. A subset of $\rea^{n}$ is \textit{(real) algebraic} if it is the zero set of a polynomial, $u$. The nonsingular locus of a real algebraic set comprises the points where the differential $Du$ is nonzero. A subset $M \subset \sphere^{n-1}$ is \textit{algebraic (of degree $k$)} if the cone $\cone(M)$ is contained in the zero level set of a degree $k$ homogeneous polynomial, $u$, on $\rea^{n}$. A real algebraic subset $M \subset \sphere^{n-1}$ is \textit{minimal} if its nonsingular locus is minimal.

Here as in the rest of the paper, for expository simplicity, all algebraic structures are supposed defined over the real field, $\rea$, and real algebraic is abbreviated as algebraic.

For example, for $p \geq q$, the set of $p \times q$ matrices $X$ satisfying $\det X^{t}X = 0$, is a minimal algebraic submanifold of the sphere $\{X: \tr X^{t}X = 1\}$ \cite{Tk10c}, \cite{Bordemann-Choe-Hoppe}.

For sufficiently smooth $F:\rea^{n} \to \rea$ define a differential operator by
\begin{align}
\begin{aligned}
\minop(F) &= |D F|^{2}\lap F - \tfrac{1}{2}\lb D F, D|D F|^{2}\ra 
=F^{p}F_{p}F_{q}\,^{q} - F_{pq}F^{p}F^{q}.&
\end{aligned}
\end{align}
A straightforward calculation shows that at a regular point of a level set $F^{-1}(c)$, $c\in R{}$, the mean curvature equals $\tfrac{1}{n-1}|DF|^{-3}\minop(F)$.

If $F:\rea^{n} \to \rea$ is an irreducible homogeneous polynomial, it follows from some version of the real algebraic nullstellensatz (see \cite[Lemma $2.2$]{Perdomo2013}), that the cone $F^{-1}(0)$ is minimal if and only if there is a homogeneous polynomial $Q:\rea^{n} \to \rea$ such that $\minop(F) = QF$. This relation can be summarized schematically as
\begin{align}\label{hsianganalytic}
\text{the cone $F^{-1}(0)$ is minimal} \iff \minop(F) = 0 \,\,\text{mod}\,\,F.
\end{align}
Because $\minop(GF) = G^{3}\minop(F) \,\,\text{mod}\,\,F$, when looking for polynomial solutions to $\minop(F) = QF$ it is reasonable to restrict attention to irreducible polynomials.
The homogeneous linear solutions of \eqref{hsianganalytic} are simply hyperplanes, which are cones over equators in $\sphere^{n-1}$.
If $F$ is homogeneous of degree $k \geq 2$ then $\minop(F)$ is homogeneous of degree $3k-4$, so $Q$ must be homogeneous of degree $2k-4$. In the case $k = 2$, $Q$ is a constant, and Hsiang showed in \cite{Hsiang67} that in this case any solution of \eqref{hsianganalytic} is equivalent to a multiple of
\begin{align}
&(q-1)(x_{1}^{2} + \dots x_{p}^{2}) - (p-1)(x_{p+1}^{2} + \dots + x_{p+q})^{2}, &&p, q\geq 2,\quad p+q = n.
\end{align}
The corresponding minimal hypersurfaces in $\sphere^{n-1}$ are known as Clifford tori.

Because of the equivariance property, $\minop(g\cdot F) = g\cdot \minop(F)$ for $g \in O(n)$ (where $(g\cdot F)(x) = F(g^{-1}x)$)), a special case of $\minop(F) = 0 \,\,(\text{mod}\,\, F)$ is obtained by requiring $\minop(F)/F$ to have some prespecified form, for example that $\minop(F)/F$ equal some \emph{orthogonally invariant} polynomial. When $k = 3$, then $Q$ must be a quadratic form. An orthogonally invariant quadratic form must be a multiple of the Euclidean quadratic form $h(x,x)$.

For cubic homogeneous polynomial solutions of \eqref{hsianganalytic} (the normalization factor is the same as in \eqref{1laplace} will be explained latter) there result two variants of the \emph{Hsiang} equations:
\begin{align}\label{nonradialhsianganalytic}
\minop(u) &= -\frac32 b(x,x)\,u && \text{nonradial Hsiang}\\
\label{radialhsianganalytic}
\minop(u) &= -\frac32\theta h(x,x)u && \text{radial Hsiang}
\end{align}
where $Q$ is an arbitrary (undetermined) quadratic form and $\la \in \rea$ is a constant. Here \emph{nonradial} means \emph{not necessarily radial}, and evidently any solution to the radial Hsiang equations solves also the nonradial Hsiang equations.

The equations \eqref{nonradialhsianganalytic}-\eqref{radialhsianganalytic} are called \emph{Hsiang} in honor of Wu-yi~Hsiang who first found interesting cubic solutions in \cite{Hsiang67}. More precisely, Hsiang observed that a homogeneous minimal hypersurface in a sphere is real algebraic and constructed cubic solutions to \eqref{radialhsianganalytic} to show that a minimal real algebraic hypersurface in a sphere need not be algebraic. He asked whether \eqref{hsianganalytic} admits irreducible homogeneous polynomial solutions of arbitrary degree and to classify cubic homogeneous polynomial solutions. It is this last question that is broached here.

\begin{definition}
Following the earlier work \cite{Tk10c} and \cite[Chap.~6]{NTVbook}, cubic solutions to the (radial) Hsiang equations are called \textit{(radial) eigencubics}.
Two radial eigencubics are \emph{equivalent} if they are on the same orbit of the action of the conformal group $CO(n)$.
\end{definition}

On $\rea^{3}$ any product of three linear forms solves \eqref{nonradialhsianganalytic}. Such a product solves \eqref{radialhsianganalytic} if the three forms are orthonormal or are all equal, but not in general. This shows that there are solutions of the nonradial Hsiang equations that are not radial.

Initially the radial case appears simply as a technical simplification that facilitates finding solutions.
It is not clear how big the gap between the nonradial and radial cases is. Comparison with the special case of pseudocomposition/eikonal algebras (see Example \ref{ex:cart0} below) suggests that it is not large. Moreover, every example we know of an \textit{irreducible} solution of the nonradial Hsiang equation in fact solves the radial Hsiang equation.

\subsection{Clifford type vs exceptional}
A cubic polynomial solution $u(x)$ of \eqref{1laplace} is called \emph{degenerate} if $u(x) = \om(x)^{3}$ for some homogeneous linear form $\om(x)$.
In \cite{Tk10b}, the second author proved that any cubic polynomial solution $u(x)$ of \eqref{1laplace} is either degenerate or it must be a harmonic function, $\Delta u(x)=0$. The latter condition simplifies considerably the original equation \eqref{1laplace}. In \cite{Peng}, C.K.~Peng and X.~Liang took an approach to the classification of radial eigencubics similar to that of \cite{Tk10b}, but under the a priori assumption that $u(x)$ be harmonic.

In \cite{Tk10c}, an infinite family of solutions to \eqref{1laplace} based on representations of Clifford algebras was constructed. More precisely, there is one-to-one correspondence between the (equivalence classes of) \textit{symmetric Clifford systems} \cite{FKM}
$A=(A_1,\ldots,A_q)\in \End(\R{2p})^q$ satisfying
\begin{equation}\label{AA}
A_i^t=A_i, \,\, A_iA_j+A_jA_i=2I\delta_{ij},\quad \forall 1\le i,j\le q,
\end{equation}
and radial eigencubics given explicitly by
\begin{equation}\label{uA}
u_{A}(x)=\sum_{k=1}^{q}x_{k+2p}\, \sum_{i,j=1}^{2p}A_{k,ij}x_ix_j, \qquad x\in \R{2p}\times \R{q}\cong \R{2p+q}.
\end{equation}
\begin{definition}
A radial eigencubic has \emph{Clifford type} if it is equivalent to one of the form \eqref{uA}.
\end{definition}

\begin{example}
The simplest example is the Lawson minimal cone \cite{Lawson} determined by
$$
u(x)=x_1x_2x_3+\frac{1}{2}(x_1^2-x_2^2)x_4, \quad x\in \R{4},
$$
for which the Clifford data is $A_1x_1+A_2x_2=\begin{pmatrix}x_1 & x_2 \\x_2 & -x_1\\\end{pmatrix}$.
\end{example}

It is well known \cite{Shapiro} that there exists a symmetric Clifford system \eqref{AA} if and only if $(q,p)\in \mathbb{N}^2$ satisfy
\begin{equation}\label{rhommm}
q-1\le \rho(p),
\end{equation}
where  the Hurwitz-Radon function $\rho$ is defined by $\rho(m)=8a+2^b$, if $m=2^{4a+b}\cdot \mathrm{odd}$, $0\leq b\le 3$.


\begin{example}\label{ex:tri}
A \emph{composition algebra} is an algebra $(\Alg,\diamond)$ equipped with a nondegenerate quadratic form $n:\Alg\to \R{}$ that is multiplicative, meaning
\begin{align}\label{compos}
n(x\diamond y)=n(x)n(y).
\end{align}
A \emph{Hurwitz algebra} is a \textit{unital} composition algebra with unit $e$. It is equipped with the involution $\bar{x} = n(e, x)e - x$, where $n(x, y)$ is the polar form obtained by linearizing $n(x)$ (so $n(x, x) = 2n(x)$). By the Hurwitz theorem \cite{Albert46, SpringerVeldkamp}, a Hurwitz algebra over a field $\fie$ has dimension in $\{1, 2, 4, 8\}$ and is isomorphic to the base field $\fie$, a quadratic étale $\fie$-algebra, a quaternion $\fie$-algebra, or an octonion $\fie$-algebra. In particular a real Hurwitz algebra with positive definite quadratic form is isomorphic to the real field $\Hurw_1=\rea$, the complex field $\Hurw_2=\com$, the quaternions $\Hurw_4=\quat$, or the octonions $\Hurw_8=\cayley$. Note that there do exist nonunital composition algebras; the simplest example is the \emph{para-Hurwitz algebra} obtained by equipping a Hurwitz algebra $(\alg, \diamond)$ having $d \geq 2$ with the twisted multiplication $u\star v = \bar{u}\diamond\bar{v}$. Another interesting example of a nonunital composition algebra is the Okubo or pseudooctonion algebra, an $8$-dimensional nonassociative algebra originally introduced by Susumu Okubo in \cite{Okubo78}. See \cite{Elduque2020} for a survey.

An important subclass of Clifford type radial eigencubics is those defined by
\begin{align}\label{cartan}
u_{\Hurw_d}(x):=\re ((z_1z_2)z_3), \quad x=(z_1,z_2,z_3)\in \Hurw_{d}^{3} \simeq \rea^{3d},
\end{align}
where $\Hurw_d$ is a real Hurwitz algebra of dimension $d$ and $z_k=(x_{kd - (d-1)},\ldots, x_{kd})\in \R{d}\cong\Hurw_d$, $k\in\{1,2,3\}$. The parentheses are necessary in \eqref{cartan} because the octonions $\Hurw_{8}$ are not associative. Example~\ref{ex:trial} below gives a conceptually simple proof that these cubic forms are indeed eigencubics that is based on the triple algebra construction of Definition \ref{def:triple}.
\end{example}

\begin{definition}
A radial eigencubic is a \emph{Hurwitz eigencubic} if it is equivalent to one of the cubics \eqref{cartan}.
\end{definition}

The Hurwitz eigencubics \eqref{cartan} appear in connection with the classical work of E.~Cartan \cite{Cartan38, Cartan39MZ} classifying isoparametric hypersurfaces in the spheres with three distinct principal curvatures \cite{CecilAnnals}.
More precisely, Cartan established that the cone over such an isoparametric hypersurface is the zero locus in $\R{3d+2} = \fie_{d}^{3}\times \rea^{2}$ of the cubic homogeneous polynomial (using the notation in \eqref{cartan})
\begin{equation}\label{CartanFormula0}
    \begin{split}
u_{J_d}(x) =x_{3d+2}^3&+\frac{3}{2}x_{3d+2}(|z_1|^2+|z_2|^2-2|z_3|^2-2x_{3d+1}^2)\\
&+\frac{3\sqrt{3}}{2}x_{3d+1}(|z_2|^2-|z_1|^2)+{3\sqrt{3}}\re ((z_1z_2)z_3),
\end{split}
\end{equation}
where $x=(z_1,z_2,z_3,x_{3d+1},x_{3d+2}) \in \fie_{d}^{3}\times \rea^{2}$. For $d \geq 1$, the cubic form \eqref{CartanFormula0} is the restriction to the subspace of trace-zero matrices of the generic determinant on the Jordan algebra $J_d$ of Hermitian $3\times3$-matrices over $\Hurw_{d}$. Retracting in a appropiate sense onto the subspace of matrices having zero diagonal yields the Hurwitz eigencubics \eqref{cartan}, which appear as the final term in \eqref{CartanFormula0}. Although the isoparametric description requires interpretation when $d = 0$, the formula \eqref{CartanFormula0} dos not, and it makes sense to consider the cubic polynomial $u_{J_d}$ for $d \in \{0, 1, 2, 3, 4\}$. In particular, it is harmonic, $\Delta u_{J_d}(x)=0$, and satisfies the Cartan-M\"unzner equation
\begin{equation}\label{Cartan2}
|D u_{J_d}(x)|^2=9|x|^4.
\end{equation}
From \eqref{Cartan2} it follows immediately by the Euler theorem for homogeneous functions that $u_{J_d}(x)$ satisfies \eqref{1laplace}, so is a radial eigencubic.
The cubic form \eqref{CartanFormula0} for $d=1$ was used in the construction of truly viscosity solutions to a fully nonlinear PDE in $\R{5}$ in \cite{NTV, NTVbook, NV2011a} (see also the appearance of $u_{J_d}$ in the context of integrable systems in \cite{Ferapontov}).

\begin{definition}
A radial eigencubic has \emph{Cartan type} if it is equivalent to one of the the \emph{Cartan eigencubics} $u_{J_d}$ in \eqref{CartanFormula0} for $d \in \{0, 1, 2, 3, 4\}$.
\end{definition}

\begin{proposition}
The Cartan eigencubics are \emph{not} of Clifford type.
\end{proposition}

\begin{proof}
The squared norm of the gradient $|Du(x)|^2$, $x\in \R{n}$ is invariant under the action of the orthogonal group $O(n)$. For a Clifford type eigencubic $u_{A}(x)$, the quartic polynomial $|Du_A(x)|^2$ has degree 2 as a polynomial in the variables $x_{2p+k}$, $1\le k\le d$, whereas, by \eqref{Cartan2}, $|D u_{J_d}(x)|^2=9|x|^4$ is a degree 4 polynomial in all its variables. Because $|x|^4$ is fixed by the action of the orthogonal group it has degree four in every variable in any orthogonal coordinate system so cannot be equivalent to $|Du_A(x)|^2$.
\end{proof}

This proposition suggests a natural dichotomy for radial eigencubics.

\begin{definition}
A radial eigencubic is \emph{exceptional} if it is not of Clifford type.
\end{definition}

This definition is by negation and so does not indicate how to determine the exceptionality of a concrete eigencubic. Theorem \ref{th:exc1} gives an effective characterization of a very slightly larger class of eigencubics.

\begin{theorem}[{\cite[Theorem~4]{Tk10b}}]\label{th:exc1}
The quadratic form $(\trace \hess u(x))^2$ is proportional to the Euclidean norm $|x|^2$ if and only if $u(x)$ is either an exceptional radial eigencubic or one of the four Hurwitz eigencubics.
\end{theorem}

Theorem \ref{th:trace} below proves an algebraic claim equivalent to Theorem \ref{th:exc1}.

The Clifford type radial eigencubics form a well understood family with an effective classification due to the classical structural results about Clifford modules due to Atiyah-Bott-Shapiro \cite{ABS}, while the class of exceptional eigencubics is more subtle to characterize (the situation is somewhat analogous to that of exceptional Lie algebras). Despite progress described in \cite[Ch.~6]{NTVbook}, \cite{Tk10b, Tk19a, Tk18e}, the classification of exceptional eigencubics is not complete, although work in this direction is ongoing.
On the other hand, it has been established that exceptional eigencubics can exist only in \emph{finitely many dimensions} $n$ no greater than $72$ \cite[Theorem 6.11.2, p.~201]{NTVbook}.

Table~\ref{tabs} below summarizes the possible ambient dimensions $n$ together with the invariant Peirce dimensions $n_{1}$ and $n_{2}$ defined later in section~\ref{sec:Peirce} (see in particular \eqref{peircedimensions}). For most of these dimensions existence has already been analyzed. The first five entries with $n_2=0$ correspond to the Cartan eigencubics \eqref{CartanFormula0}.

\begin{small}
\def\MM{6mm}
\begin{table}[ht]
\renewcommand\arraystretch{1.5}
\noindent
\begin{flushright}
\begin{tabular}{p{6.5pt}|p{6.5pt}|p{6.5pt}|p{6.5pt}|p{6.5pt}|p{6.5pt}|
p{6.5pt}|p{6.5pt}|p{6.5pt}|p{6.5pt}|p{6.5pt}|p{6.5pt}|p{6.5pt}|p{6.5pt}|p{6.5pt}|p{6.5pt}|p{6.5pt}|p{6.5pt}|p{6.5pt}|p{6.5pt}|p{6.5pt}|
p{6.5pt}|p{6.5pt}|p{6.5pt}|p{6.5pt}}
$n$  &$2$ &  $5$ & $8$  & $14$  & $26$  &$9$ & $12$ & $\textcolor{black}{15}$ & $21$ & $15$ & $18$ & $\textcolor{black}{21}$ & $\textcolor{black}{24}$ &  $\textcolor{black}{30}$&  $\textcolor{black}{42}$& $27$ & $30$  & $\textcolor{black}{33}$  & $\textcolor{black}{36}$  & $\textcolor{black}{51}$ & $54$ & $\textcolor{black}{57}$ & $\textcolor{black}{60}$ & $\textcolor{black}{72}$      \\\hline
$n_1$& $1$ &$2$ & $3$  & $5$  & $9$  & $0$ & $1$  & $2$  & $4$  & $0$  & $1$  &  $2$ & $3$ &  $5$ & $9$  & $0$  & $1$   &  $2$  & $3$   & $0$  & $1$  & $2$ & $3$ & $7$ \\\hline
$n_2$&$0$ &  $0$ & $0$  & $0$  & $0$  &$5$ & $5$  & $5$  & $5$  & $8$  & $8$  &  $8$ &  $8$ &  $8$ & $8$  &$14$  & $14$  & $14$  & $14$  & $26$ & $26$ & $26$ & $26$ & $26$\\\hline
$d$&$-$ &  $-$ & $-$  & $-$  & $-$  &$1$ & $1$  & $1$  & $1$  & $2$  & $2$  &  $2$ &  $2$ &  $2$ & $2$  &$4$  & $4$  & $4$  & $4$  & $8$ & $8$ & $8$ & $8$ & $8$\\
\end{tabular}
\end{flushright}
\bigskip
\caption{The possible Peirce dimensions $(n_1,n_2)$ of exceptional algebras.}\label{tabs}
\end{table}
\end{small}

\begin{remark}
The dimensions $n_{1}$, $n_{2}$, and $n$ appearing in Table \ref{tabs} are \emph{potentially} realizable in the sense that there is not in hand a theorem precluding the existence of Hsiang algebras with these dimensions. They are obtained by using certain obstructions on an arbitrary Hsiang algebra; for a detailed explanation, the reader is referred to \cite{NTVbook}. On the other hand, a more delicate analysis given in the unpublished manuscript \cite{Tk16} shows, for example, that the Peirce dimensions $(n_1,n_2)\in \{(2,5), (2,8),\,(2,14),\,(3,8),\, (3,14),\, (7,14)\}$ are \emph{not realizable} in the sense that there do not exist Hsiang algebras with these Peirce dimensions. The search of an alternative approach for analyzing exceptional Hsiang eigencubics (Hsiang algebras) and the determination of realizable Peirce dimensions was among of the original motivations for the triple construction considered in the present paper.
\end{remark}

\begin{remark}
Theorem \ref{th:exc1} shows that the PDE $(\tr \hess u(x))^{2} = c|x|^{2}$ exhibits a certain rigidity in that it admits quite few homogeneous cubic polynomial solutions. This should be constrasted with the PDE
\begin{align}\label{killingpde}
|\hess u(x)|^{2} = c|x|^{2}
\end{align}
studied by the first author in \cite{Fox2021a} where it is shown to admit a wide array of infinite families of homogeneous cubic polynonial solutions. The PDE \eqref{killingpde} corresponds to the Killing metrizability of the metrized commutative algebra associated with $u$ in the sense defined in section \ref{sec:Kill}.
\end{remark}

\section{Metrized commutative algebras}\label{sec:metr}

The first systematic use of nonassociative algebra to approach the classification of cubic solutions of the Hsiang equations was made in \cite[Chap.~6]{NTVbook}. In \cite{Fox2009, Fox2015, Fox2021a} a similar algebraic approach was taken to finding cubic solutions of \eqref{killingpde}. Before describing (see Section~\ref{sec:Hsiang}) this approach we outline the basis notions of nonassociative algebra that are needed.

\subsection{The general case}
An \emph{algebra} is a vector space $\Alg$ equipped with a bilinear map $\diamond: \Alg \times \Alg \to \Alg$. The product $\diamond$ is \textit{not} required to be unital, commutative, or associative.

The left and right multiplication endomorphisms $L_{\diamond}, R_{\diamond}: \Alg \to \eno(\Alg)$ are defined by
$L_{\diamond}(x)y= x\diamond y$ and $R_{\diamond}(x)y= y\diamond x$, respectively. When it is obvious from context the subscript indicated the multiplication is omitted and there is written simply $L(x)$ or $R(x)$.

A emph{unit} of $\alg$ is an element $e\in \Alg$ such that $e\diamond x=x\diamond e=x$ for all $x\in \Alg$. An algebra is \emph{unital} if it has a unit, in which case the unit is unique.

An \emph{involution} $\sigma$ of an algebra $\Alg=(V,\diamond)$ is a linear endomorphism $\sigma\in \End(\Alg)$ satisfying $(x\diamond y)^\sigma=y^\sigma\diamond x^\sigma$ for all $x, y \in \Alg$ and squaring to the identity $\sigma\circ\sigma=1$. By definition, in an algebra with involution there hold
\begin{align}\label{RL}
R(x^{\si})=\sigma \circ L(x)\circ \sigma, \qquad
L(x^{\si})=\sigma \circ R(x)\circ\sigma.
\end{align}
In particular, for an algebra with involution, either of $L(x)$ and $R(x)$ determines the other.
The identity operator $\sigma=1$ on a commutative algebra and the negative of the identity $\sigma=-1$ on a anticommutative algebra are obviously involutions; these called the standard involutions. Note that there are both commutative and anticommutative algebras admitting nonstandard involutions and an algebra with involution need not be either commutative or anticommutative.

A bilinear form $h(x,y)$ on a vector space $\alg$ is symmetric if $h(x,y)=h(y,x)$ and a symmetric bilinear form $h$ is nondegenerate if its radical is trivial, in other words, $h(x,y)=0$ for any $y\in \alg$ implies $x=0$. A nondegenerate symmetric bilinear form is called a \emph{metric}. A metric is \emph{Euclidean} if it is positive definite.

\begin{definition}
A symmetric bilinear form $h(x,y)$ on an algebra $\Alg$ with involution $\sigma$ is \textit{involutively invariant} if
\begin{align}
h(x^\sigma,y^\sigma)&=h(x,y),\label{invol}\\
h(x\diamond y,z)&=h(x,z\diamond y^\sigma), \quad \forall x,y,z\in \Alg.\label{Qass}
\end{align}
\end{definition}
The condition \eqref{invol} states that $\si$ is $h$-orthogonal. Under condition \eqref{invol}, the relation \eqref{Qass} is equivalent to
\begin{align}\label{Qass1}
h(R(y)x, z) & = h(x\diamond y,z)=h(x,z\diamond y^\sigma) = h(x, R(y^{\sigma})z), \quad \forall x,y,z\in \Alg,
\end{align}
which states that the $h$-adjoint $R(y)^{\ast}$ equals $R(y^{\si})$. Likewise, \eqref{Qass1} together with \eqref{RL} implies that the $h$-adjoint $L(x)^{\ast}$ satisfies
\begin{align}\label{self}
L(x)^{\ast} =  L(x^{\si})=\sigma \circ R(x)\circ \sigma.
\end{align}

\begin{definition}
An algebra with involution is \emph{metrized} if it is equipped with an involutively invariant metric. A commutative algebra is said to be \emph{metrized} if it is equipped with a metric involutively invariant with respect to the standard involution.
\end{definition}

Metrized algebras with involution abound. Some interesting examples include:
\begin{itemize}
\item A semisimple Lie algebra with its standard involution is metrized by its Killing form $h(x,y)=\trace(\mathrm{ad}(x)\mathrm{ad}(y))$.
\item Many solvable and nilpotent Lie algebras with vanishing Killing form nonetheless admit invariant metrics. 
\item A Hurwitz algebra (a unital composition algebra) is metrized by the trace form $h(x, y) = \tr L(\bar{x})L(y)$ where $x \to \bar{x}$ is the canonical involution.
\item The commutative Griess algebras of certain vertex operator algebras, in particular \emph{the} Griess algebra on which the Monster finite simple group acts by automorphisms, are naturally metrized, as are many related Norton and axial algebras \cite{Norton94}, \cite{HRS15}, \cite{Ivanov15}.
\end{itemize}

\begin{example}[Jordan algebras]\label{ex:Jord}
Recall that a \emph{Jordan algebra} is a commutative algebra $(A,\diamond)$ satisfying the commutator relation
\begin{align}\label{jordan}
[L(x),L(x^2)]=0.
\end{align}
This identity can be thought as a weakening of the associative law.
In particular, any Jordan algebra is power-associative (the subalgebra generated by a single element  is always associative). On any semisimple Jordan algebra $(A,\diamond)$ the bilinear form
$$
h(x,y):=\trace L(x\diamond y),
$$
is nondegenerate and invariant \cite[p. 60]{Koecher}. In other words, $(A,\diamond)$ is naturally metrized by $h$.

The fundamental paper \cite{JordanNeumann} classified finite-dimensional semisimple Euclidean Jordan algebras and showed that any such is a direct sum of simple building blocks, namely, the spin factors (Jordan algebras satisfying a quadratic relation), Hermitian matrices of size $n \geq 3$ over a Hurwitz algebra $\Hurw_d$, for $d=1,2,4$, and the exceptional Albert algebra of $3\times 3$ Hermitian matrices over octonions $\Hurw_8$; in these cases the Jordan product is the symmetrization of the standard matrix product: $x\diamond y=\frac12(xy+yx)$. This classification is an essential ingredient in the classification of radial eigencubics because these have a certain hidden Jordan algebra structure (see the remark following \eqref{bullet1}).
\end{example}

\subsection{Commutative metrized algebras}
This section considers metrized commutative algebras in more detail. Let $(\alg,\diamond)$ be a commutative algebra equipped with the standard involution $\sigma=1$ metrized by $h$. In this case \eqref{invol} is trivial, while \eqref{Qass} is equivalent to the complete symmetry (invariance under permutations) of the trilinear form $t$ defined by
\begin{align}
\label{trilinear}
t(x,y,z):=h(x\diamond y, z).
\end{align}
This complete symmetric reflects an important property of the multiplication operator in a metrized commutative algebra $(\Alg,\diamond,h)$. Namely, by \eqref{self}, $L(x)$ is an $h$-\textit{self-adjoint operator}.

Associate with $\Alg$ the \textit{cubic form}
\begin{align}
\label{cubic}
u(x):=\frac16t(x,x,x):=\frac16h(x\diamond x, x).
\end{align}
In the converse direction, given a metric cubic space $(\alg, h, u)$, meaning a cubic form $u(x)$ on a vector space $\alg$ with a nondegenerate symmetric bilinear form $h$, a commutative algebra structure is determined as follows: given a pair $x,y\in \alg$, define $x\diamond y$ as the unique vector satisfying
\begin{align}
\label{trilinear1}
h(x\diamond y, z)=t_u(x,y,z), \qquad \forall z\in V,
\end{align}
where
\begin{align}\label{full3}
t_u(x,y,z):=u(x+y+z)-u(x+y)-u(x+z)-u(y+z)+u(x)+u(y)+u(z)
\end{align}
is the complete polarization of $u(x)$. The product $\diamond$ is evidently commutative, and the corresponding algebra $(\alg, \diamond,h)$ is metrized.

The correspondence between metrized commutative algebras $(\alg,\diamond, h)$  and cubic forms on an inner product vector space $(\alg,h)$ is functorial with respect to the natural notions of morphisms. In particular, isometrically isomorphic metrized commutative algebras correspond to isomorphic metric cubic spaces. This means that the study of metric cubic spaces is equivalent to the study of commutative metrized algebras.

This correspondence means that analytic calculus of cubic forms translates in a nonassociative algebra language as follows. Observe that
\begin{align}\label{u41}
\frac12t(x,x,y)=\partial_y u|_{x}
\end{align}
is the directional derivative of $u$ in the direction $y$ evaluated at $x$ which is a quadratic with respect to $x$ and linear with respect to $y$. On the other hand, the gradient of $u$ (with respect to the inner metric $h$ on $\Alg$) is defined as the unique vector $\nabla u(x)$ such that $\partial_y u|_{x}=h(\nabla u(x),\, y)$ holds for any $y\in \Alg$. Therefore the Euler homogeneity function theorem combined with \eqref{trilinear} implies
\begin{align}\label{grad}
\nabla u(x)=\frac12 x\diamond x,
\end{align}
and by linearization we obtain
\begin{align}\label{hess}
\hess u(x)(y, z)= h(L_{\diamond}(x)y, z).
\end{align}
The condition that $u$ be harmonic, $\Delta u=0$, is equivalent to the trace free condition
\begin{align}\label{exact}
\Delta u(x)=\trace L_{\diamond}(x)=0, \qquad \forall x\in \Alg.
\end{align}
Following \cite{Fox2021a}, algebras satisfying \eqref{exact} are called \emph{exact} (they were called \emph{harmonic} in \cite{NTVbook}).

The relations \eqref{hess} and \eqref{exact} determine the translation from the PDE setting to the nonassociative language, and vice versa. For example, by \eqref{grad}, the singular locus of the zero level set of $u$ comprises the $2$-nilpotent elements (those satisfying $x\diamond x=0$) in the associated metrized commutative algebra. As will be seen later the different types of radial eigencubics reflected in the Diagram \ref{diagram} correspond strongly with the structure of their singular loci.

Note that over an arbitrary field a metrized commutative algebra need not contain a nonzero idempotent. On the other hand, one can prove in a purely analytic manner that a Euclidean metrized commutative algebra always contains at least one nonzero idempotent (it cannot be guaranteed that there is more than one idempotent; an example is in \cite[section~3]{Tk18a}).

\begin{proposition}[\cite{Tk18a}]\label{pro:ide}
Any nonzero Euclidean metrized commutative algebra $(\Alg,\diamond,h)$ contains a nonzero idempotent.
\end{proposition}

\begin{proof}
Since $(\Alg,\diamond,h)$ is a nonzero algebra, its cubic form $u(x)=\frac16 h(x^2,x)$ is not identically zero. The sphere $S=\{x\in \Alg:\scal{x}{x}=1\}$ is compact in the Euclidean topology on $\Alg$, hence the cubic form $u(x)$ attains its maximum on $S$, say at $z\in S$. Because $u$ is not identically zero, $u(z)>0$. The Lagrange multiplier condition is $\nabla u(z)=k\nabla \scal{z}{z}=2kz$ for some $k\in \R{}$. On the other hand,  by the Euler homogeneous function theorem $u(z)=\frac13 \scal{\nabla u(z)}{z}=\frac43k\scal{z}{z}>0$ implying $k\ne0$. Using \eqref{grad} this yields $z\diamond z=4kz$, hence $c:=z/4k$ satisfies $c\diamond c=c$, that is $c$ is a nonzero idempotent in $\Alg$, as desired.
\end{proof}


\section{From Hsiang eigencubics to Hsiang algebras}\label{sec:Hsiang}

The correspondence between metrized cubic spaces and metrized commutative algebras described in section \ref{sec:metr} is use to translate the Hsiang equation into the definition of Hsiang algebras.

\subsection{Definition of Hsiang algebras}
In this section $\alg=\R{n}$ and $h(x,y)=\scal{x}{y}$ is the Euclidean scalar product. For readability, $L_{\diamond}(x)$ is written $L(x)$ and $\diamond$-powers of an element $x$ are written as juxtapositions,
$$
x^{2} = x\diamond x, \quad
x^{3} = x^{2}\diamond x = (x\diamond x)\diamond x,
\quad x^{2} x^{2} = (x\diamond x)\diamond(x\diamond x),\quad \text{ etc.}
$$
This notation is unambiguous because powers of order 2 and 3 are well defined in any commutative algebra.

Suppose that the cubic form $P(x) = \tfrac{1}{6}h(x^2, x)$ of the metrized commutative algebra $(\alg, \diamond, h)$ solves $\minop(P) = 0 \,\,(\text{mod}\,\, P)$. Rewriting this using \eqref{grad} and \eqref{hess} yields
\begin{align}
h(x^2, x^3 )- h(x^2, x^2)\tr L(x) = 0\,\,\text{mod}\,\, P.
\end{align}
If $P$ is irreducible it follows that there is a symmetric bilinear form $b(x, y)$ such that
\begin{align}\label{nonradialhsiang}
h(x^{2} , x^{3})  - h(x^{2} , x^{2} ) \tr L(x) = b(x, x)h(x, x^{2} ).
\end{align}
The following definition generalizes one from \cite{NTVbook}.
\begin{definition}\label{def:Hsiang}
A Euclidean metrized algebra $(\alg,\diamond,h)$ is a \emph{Hsiang algebra} if it satisfies \eqref{nonradialhsiang} for some symmetric bilinear form $b$ on $\Alg$. If there is a constant $\theta \in \rea$ such that $b = \theta h$, then $(\alg,\diamond,h)$ is said to be a \emph{radial Hsiang algebra}.
\end{definition}

The radial Hsiang algebras are characterized by the equation
\begin{align}\label{radialhsiang}
h(x^{2} , x^{3})  - h(x^{2} ,x^{2} ) \tr L(x) = \theta h(x, x) h(x, x^2).
\end{align}

\begin{remark}
Replacing $h$ by $r^{-1}h$ replaces $\theta$ by $r\theta$, so $\theta$ can be normalized as one likes by rescaling $h$. In many applications, the choice $\theta=4/3$ is natural. A radial Hsiang algebra satisfying $\theta=4/3$ is said to be \emph{normalized}.
\end{remark}

\begin{proposition}
A cubic form $u(x)$ solves \eqref{1laplace} if and only if the commutative metrized algebra $(\R{n},\diamond,\scal{}{})$ with $\diamond$ induced by $u$ as in \eqref{trilinear1} is a radial Hsiang algebra.
\end{proposition}

If the cubic form of a Euclidean radial Hsiang algebra $(\alg, \diamond, h)$ equals the cube of a linear form, $\om(x)^{3}$, then the corresponding product $\diamond$ satisfies $x \diamond y = 6\om(x)\om(y)\alpha$ for the element $w\in \Alg$ uniquely determined by $\om(x) = h(x, w)$. In this case  $\tr L_{\diamond}(x) = 6h(w, w)\om(x)$. These conditions in fact characterize radial Hsiang algebras whose cubic forms are cubes of linear forms.
\begin{theorem}\label{th:degenerate}
For a Euclidean radial Hsiang algebra the following are equivalent.
\begin{enumerate}
\item It is not exact.
\item $\dim(\Alg\diamond \Alg) = 1$.
\item Its cubic form $u(x)$ equals the cube of a linear form.
\end{enumerate}
\end{theorem}
The proof, which is more complicated than one might expect, is omitted (see however the proof in \cite[section~6.6.3]{NTVbook}). The key points are the following. First, by Proposition~\ref{pro:ide}, a Euclidean metrized commutative algebra contains a nonzero idempotent. Second, if a radial Hsiang algebra is not exact, there is a unique element $\ell$ such that $h(\ell, x) = \tr L(x)$ and it follows from \eqref{radialhsiang} that any idempotent is a multiple of $\ell$. From this it can be deduced that $\Alg\diamond \Alg$ is spanned by $\ell$.

\begin{definition}
A radial Hsiang algebra is \emph{degenerate} if it satisfies the equivalent conditions of Theorem \ref{th:degenerate} and \emph{nondegenerate} otherwise.
\end{definition}

Theorem \ref{th:degenerate} justifies the assertion in section \ref{sec:prelim} that  for a nondegenerate solution $u(x)$ of \eqref{1laplace}, there holds $\Delta u(x)=0=\trace L(x)$. By Theorem \ref{th:degenerate}, a nondegenerate radial Hsiang algebra is exact, and in this case \eqref{radialhsiang} specializes to the equations $\tr L(x) = 0$ and
\begin{align}
\label{Hsiang1}
h(x^2,\,x^3)=
\theta h(x,x)\,h(x^2,\, x). 
 \end{align}
The linearization of \eqref{Hsiang1} yields the identity
\begin{align}
\label{Hsiang2}
4x^3x+ x^2x^2- 2\theta h(x,x) x^2- 3\theta h(x^2,\, x)x=0,
\end{align}
while via the invariance of $h$ the identity \eqref{Hsiang2} implies \eqref{Hsiang1}.

\begin{corollary}
A cubic form $u(x)$ is a nondegenerate solution of \eqref{1laplace} if and only if the commutative metrized algebra $(\R{n},\diamond,\scal{}{})$ with $\diamond$ induced by $u$ is an exact radial Hsiang algebra.
\end{corollary}

Let $c$ be a nonzero idempotent in a nondegenerate Euclidean radial Hsiang algebra. Substituting $c$ for $x$ in \eqref{Hsiang1} implies
\begin{align}
\label{length}
h(c,c)=\frac{1}{\theta}.
 \end{align}
In particular, \textit{all idempotents in a radial Hsiang algebra have the same length $\frac{1}{\sqrt{\theta}}$}.

\begin{remark}\label{rem:weak}
There are \emph{weak} versions of the equations \eqref{nonradialhsiang} and \eqref{radialhsiang} in which either of the symmetric bilinear forms $b$ or $h$ is allowed to be degenerate and interesting examples can be found. For simplicity such degenerate solutions are not discussed further here.
\end{remark}

\subsection{Polar algebras}
It is natural to identify the subclass of Hsiang algebras that corresponds with Clifford type eigencubics. This is explained next.

\begin{definition}\label{def:Clifford}
A \emph{polar algebra} is a metrized commutative algebra $(\Alg,\diamond, h)$ equipped with a nontrivial $h$-orthogonal decomposition $\Alg=\Alg_0\oplus_h \Alg_1$ (meaning that both $\Alg_0$ and $\Alg_1$ are nontrivial subspaces) such that
\begin{itemize}
\item[(i)] $\Alg_0\diamond \Alg_0=\{0\}$,
\item[(ii)] $\Alg_1\diamond \Alg_1\subset \Alg_0$,
\item[(iii)]
\begin{equation}\label{Vb}
x(xy)=h(x,x)y, \qquad \forall x\in \Alg_0,\,\forall y\in \Alg_1.
\end{equation}
\item[(iv)]
If $\dim \Alg_0=1$, then there is required additionally that $\trace L_\diamond(x)=0$ for all $x\in \Alg_0$.
\end{itemize}
\end{definition}

The following result is a corollary of the definitions (see \cite[Proposition~6.5.1]{NTVbook}).

\begin{theorem}\label{th:polar}
A polar algebra $(\Alg,\diamond, h)$ is a radial Hsiang algebra and its cubic form $u(x)=\frac16 h(x^2,x)$ is a Clifford type radial eigencubic satisfying \eqref{Hsiang1} with $\theta =4/3$. Conversely, if $u(x)$ is a Clifford type radial eigencubic satisfying \eqref{Hsiang1} with $\theta =4/3$, then the metrized algebra $(\R{n},\diamond,h)$ is isomorphic to a polar algebra. In this case there holds
\begin{align}\label{traceP}
\trace L_\diamond(x)^2=h(x_0,x_0)\dim \Alg_1+2h(x_1,x_1)\dim \Alg_0,
\end{align}
where $x={\xp}+{\xm}$ in accord with the decomposition $\Alg=\Alg_0\oplus \Alg_1$.
\end{theorem}

\begin{proof}
For completeness, the first identity in \eqref{Hsiang1} is proved (the traceless property is proved similarly).
By \eqref{Vb} for all $x\in \Alg_0$ and $y\in \Alg_1$,
\begin{align}\label{polarid1}
h(L(y)^2x,x)=h(y (y  x),x)=h(y,(y  x)  x)=h(y,y)h(x,x)=h(h(y,y)x,x).
\end{align}
By \eqref{self}, $L(y)^2$ is a self-adjoint operator, hence \eqref{polarid1} implies the restriction of $L(y)^2$ to $\Alg_0$ equals $h(y,y)1_{\Alg_0}$ on for any $y\in \Alg_1$. That is,
\begin{align}
\label{Vb1}
y  (y  x)=h(y,y)x, \quad \forall x\in \Alg_0, \,\forall y\in \Alg_1.
\end{align}
Decompose $x={\xp}+{\xm}$ according to $\Alg=\Alg_0\oplus \Alg_1$. Because $\alg_{0} \diamond\alg_{0} = \{0\}$, $x^2=\xm  \xm+2\xp  \xm,$ where $\xm  \xm\in V_0$ because $\alg_{1} \alg_{1} \subset \alg_{0}$. Since $h(V_0,V_1)=0$, using the invariance of $h$ yields
\begin{equation}\label{rhoClif}
h(x,x^2)=h(\xp, \xm  \xm)+2h(\xm,\xm  \xp)=3h(\xp,\xm  \xm).
\end{equation}
Because $\alg_{0} \alg_{0} = \{0\}$, $(\xm  \xm)  \xp=0$ and by \eqref{Vb1}
\begin{align*}
x^3&=2\xm (\xm \xp)+2\xp (\xp \xm)+(\xm  \xm)  \xm=2h(\xm,\xm)\xp+2h(\xp,\xp)\xm+(\xm  \xm)  \xm.
\end{align*}
Arguing similarly, one obtains
\begin{align*}
h(x^2,\,x^3)&
=2h(\xm,\xm)\SCAL{{\xm  \xm}}{{{\xp}}}+4h(\xp,\xp)\SCAL{{\xp}}{{\xm  \xm}}+2\SCAL{{\xm} {\xp}}{(x_1  x_1)  x_1}\\
&=4h(\xm,\mu)\SCAL{{\xm  \xm}}{{{\xp}}}+4h(\xp,\xp)\SCAL{{\xp}}{{\xm  \xm}}\\
&=\half{4}{3}h(x,x)\SCAL{x}{x^2},
\end{align*}
implying the first identity in \eqref{Hsiang1} with $\theta =4/3$. The proof of \eqref{traceP} is similar.
\end{proof}
\medskip

\begin{remark} \label{rem:polar}
By Theorem~\ref{th:polar}, the radial eigencubic $u(x)=\frac16 h(x^2,x)$ associated with a polar algebra $\Alg$ is of Clifford type with $q=\dim \Alg_0$ and $2p=\dim \Alg_1$, as follows readily from \eqref{rhoClif}.
By \eqref{rhommm} this implies
\begin{equation}\label{rho}
\dim \Alg_0-1\le \rho(\frac12 \dim \Alg_1).
\end{equation}
where $\rho$ is the Hurwitz-Radon function.
\end{remark}

\begin{remark}
The Clifford type radial eigencubics always have nontrivial singular locus. Precisely, by \eqref{grad}, since $x^2=0$ on $\Alg_0$, the subspace $\alg_{0}$ of a polar algebra is contained in the singular locus of the associated (Clifford type) radial eigencubic $u$.
\end{remark}

\begin{definition}
A Hsiang algebra is \textit{exceptional} if it is not (algebra) isomorphic  to a polar algebra.
\end{definition}

\begin{example}\label{ex:cart0}
The simplest examples of exceptional Hsiang algebras are given by pseudocomposition algebras. A commutative algebra $(\Alg,\diamond, h)$ is called a \textit{pseudocomposition algebra} \cite{Meyberg1} if it satisfies the identity
\begin{align}
\label{pseudo}
x^3=b(x, x)x
\end{align}
for some  nonzero symmetric bilinear form $b$ on $\alg$. If a metrized commutative algebra satisfes \eqref{pseudo}, then $h(x^3, \, x^2)=b(x, x)h(x,x^2)$, so that \eqref{nonradialhsiang} holds provided the algebra is moreover exact. This shows that \textit{a metrized exact pseudocomposition algebra is a nonradial Hsiang algebra}.
\end{example}

Pseudocomposition algebras naturally arise  in the context of the eiconal equation and isoparametric hypersurfaces, see \cite{Tk14}. This motivates the following

\begin{definition}
A pseudocomposition algebra is an \emph{eikonal} algebra if it satisfies \eqref{pseudo} with $b = \theta h$ for some nonzero $\theta \in \rea$ and a Euclidean metric $h$.
\end{definition}

By a theorem of Elduque and Okubo \cite[Theorem $12$]{ElOkubo}, over a field of characteristic not dividing $6$ the bilinear form $b$ of a pseudocomposition algebra is automatically invariant. By the theorem of Elduque and Okubo, for an eikonal algebra $h$ is invariant and there holds
\begin{align}\label{eikonal}
h(x^3, \, x^2)=\theta h(x, x)h(x,x^2),
\end{align}
so that \eqref{radialhsiang} holds provided the eikonal algebra is moreover exact. Therefore \emph{an exact eikonal algebra is a radial Hsiang algebra} (by the preceding remarks, in this case metrizability need not be assumed).

Note that the distinction between pseudocomposition and eikonal algebras is formally parallel to (in fact a subcase of) the distinction between nonradial Hsiang algebras and radial Hsiang algebras.

\begin{lemma}
An exact eikonal algebra with Euclidean metric is an exceptional Hsiang algebra.
\end{lemma}

\begin{proof}
Were the algebra of Clifford type there would exist a polar decomposition $\Alg=\Alg_0\oplus \Alg_1$ with subspaces $\alg_0$ and $\alg_1$ having positive dimension. For any $x\in \Alg_0$ there would hold $x^2=0$, so that $0=x^3=\theta h(x,x)x$. Because $\theta \neq 0$ and $h$ is Euclidean this implies $x =0$. This implies that $\Alg_0=\{0\}$ is trivial, a contradiction.
\end{proof}
The simplest example of an exact pseudocomposition algebra is the paracomplex numbers defined in Example \ref{ex:tri} (the paraquaternions do not give an example because they are not commutative).

Because by \eqref{eikonal} an exact eikonal algebra contains no nonzero $2$-nilpotent element, by \eqref{grad} the associated radial eigencubic has no singular locus, so is smooth.

\subsection{Types of polar algebras}\label{sec:reg}

By definition, a polar algebra is determined by a decomposition $\Alg=\Alg_0\oplus \Alg_1$. It is natural to ask whether this polar decomposition is unique in the sense that for any other polar decomposition $\Alg=\mathbb{B}_0\oplus \mathbb{B}_1$ satisfying Definition~\ref{def:Clifford} there must hold $\Alg_i=\mathbb{B}_i$, $i=0,1$.
Because the $0$-graded part of the polar decomposition is contained in the singular locus of the associated radial eigencubic $u(x)=h(x,x^2)$, this question is related to understanding the structure of the singular locus of the associated minimal cone $u(x) = 0$.

It follows from the theory of Clifford systems (or the well-known properties of the Radon-Hurwitz function $\rho$ and \eqref{rhommm}) that $\dim \alg_{1} = 2\dim \alg_{0}$ implies that $\dim \alg_{0} \in \{1, 2, 4, 8\}$. Such a polar algebra must be the algebra corresponding with one of the Hurwitz eigencubics defined in Example \ref{ex:tri}. See Corollary \ref{cor:polar} for the corresponding algebraic statement.

It follows from \cite[section~6.5.3]{NTVbook} that the uniqueness of the polar decomposition depends on the polar dimensions $(p,q)=(\dim\Alg_0,\dim\Alg_1)$. Namely, if $(p,q)\notin \{(d, 2d): d=1,2,4,8\}$ then the polar decomposition is unique, while for the remaining four pairs (Hurwitz triples), there are exactly $6=3!$ different polar decompositions (the symmetric group $S_3$ acts naturally permuting these decompositions).

The preceding motivates Definition \ref{def:mutant} and justifies that it is well made.

\begin{definition}\label{def:mutant}
A polar algebra is a \emph{mutant} if it admits a polar decomposition so that $\dim \alg_{1} = 2\dim \alg_{0}$ and is \emph{regular} otherwise.
\end{definition}

The structure of the (uniquely determined) subspace $\alg_{0}$ of a regular polar algebra is related to the singular locus of the associated radial eigencubic $u(x)$. From this point of view, polar algebras are naturally partitioned into several classes:
\begin{itemize}
\item
\emph{mutants} with $(\dim \alg_{0},\dim \alg_{1})\in \{(d, 2d): d=1,2,4,8\}$ and having radial eigencubic $u(x)$ equivalent to one of the four Hurwiz eigencubics.
\item
\emph{regular} polar algebras.
\begin{itemize}
\item It can be shown that the singular locus of $u(x)$ equals $\alg_{0}$ if and only if $\dim \alg_{0} < 2\dim \alg_{1}$.
\item When $\dim \alg_{0} > 2\dim \alg_{1}$, then the singular locus of $u(x)$ contains $\alg_{0}$ as a proper subset. Its precise structure is more complicated and will be described elsewhere.
\end{itemize}
\end{itemize}

That a mutant polar algebra admits several polar decompositions permuted by an $S_{3}$ action is a kind of triality phenomenon that can be explained by the triple algebra construction; see Corollary~\ref{cor:polar} below.


%
%

\subsection{Killing-metrized algebras}\label{sec:Kill}
Exceptional eigencubics share many properties with the mutant subclass of Clifford type radial eigencubics consisting of the four Hurwitz eigencubics, defined in Example~\ref{ex:tri}.
They are naturally united by the notion of Killing metrizability \cite{Fox2022}.

\begin{definition}
Given a metrized commutative algebra $(\Alg,\diamond,h)$, the symmetric bilinear form
$$
\kappa(x,y):=\trace L(x)L(y)
$$
is called the \textit{Killing form} of $(\Alg,\diamond,h)$. A commutative algebra $(\Alg,\diamond,h)$ is \textit{Killing metrized} if its Killing form $\kappa$ is nondegenerate and invariant.
\end{definition}

In general it is interesting that a commutative algebra admits an invariant metric.
It is a stronger condition that it be metrized by a specific trace-form, as such a form is determined by the underlying algebraic structure and so is not additional data. For an exact commutative algebra the only possibility is the Killing type form, for any nontrivial bilinear form constructed from traces of the multiplication endomorphisms and the multiplication itself must be a multiple of the Killing form.

Theorem~\ref{th:exc1} is clarified by Theorem \ref{th:trace} which is its reformulation in terms of Hsiang algebras.

\begin{theorem}\label{th:trace}
A Hsiang algebra $(\Alg,\diamond,h)$ is Killing metrized if and only if it is exceptional or mutant.
\end{theorem}

\begin{proof}
Suppose that a polar algebra $(\alg, \diamond, h)$ is Killing metrized. In this case, by \eqref{traceP} the Killing form has the expression
\begin{align}\label{kappa1}
\kappa(x,y)=\trace L(x)L(y)=h(x_0,y_0)\dim \Alg_1+2h(x_1,y_1)\dim \Alg_0,
\end{align}
where $x=(x_0,x_1)$, $y=(y_0,y_1)$ according to the polar decomposition $\Alg=\Alg_0\oplus \Alg_1$. Next, the identity \eqref{kappa1} is applied to $x=y=z_0z_1$, where $z_i\in \Alg_i$ are nonzero elements. By the assumed invariance of $\kappa$ and the relations \eqref{Vb} and \eqref{Vb1},
\begin{align}
\kappa(z_0z_1,z_0z_1)&=\kappa(z_0(z_0z_1),z_1)=h(z_0,z_0)\kappa(z_1,z_1),\\
\label{tht1}\kappa(z_0z_1,z_0z_1)&=\kappa(z_1(z_0z_1),z_0)=h(z_1,z_1)\kappa(z_0,z_0).
\end{align}
On other hand, since $z_0z_1\in \Alg_1$, by \eqref{kappa1},
\begin{align}\label{tht2}
\begin{aligned}
\kappa(z_0z_1,z_0z_1)&=2h(z_0z_1,z_0z_1)\dim \Alg_0\\
&=2h(z_0(z_0z_1),z_1)\dim \Alg_0=2h(z_0,z_0)h(z_1,z_1)\dim \Alg_0.
\end{aligned}
\end{align}
By assumption $h(z_1,z_1)\ne0$, hence combining \eqref{tht1} and \eqref{tht2} yields
\begin{align}\label{tht3}
\kappa(z_0,z_0)=2h(z_0,z_0)\dim \Alg_0.
\end{align}
On the other hand, by \eqref{kappa1}, $\kappa(z_0,z_0)=h(x_0,y_0)\dim \Alg_1$, and in \eqref{tht3} this yields $\dim \Alg_1=2\dim \Alg_0$. Applying \eqref{rho}, we get $\dim \Alg_0-1\le \rho(\dim\Alg_0)$ implying by the properties of the Radon-Hurwitz function $\rho$ that $\dim \Alg_0\in\{1,2,4,8\}$, and therefore  $(\alg, \diamond, h)$ must be a mutant polar algebra.

This shows that if a normalized Hsiang algebra is Killing metrized then it must be exceptional or mutant. That a mutant algebra is Killing metrized follows from the definition and \eqref{traceP}. That an exceptional Hsiang algebra is Killing metrized is a substantial claim proved in \cite[Proposition 6.11.1]{NTVbook} (see also \cite[Theorem 4]{Tk10b}), where the terminology is a bit different.
\end{proof}

Summarizing, there are two main classes of radial Hsiang algebras:
\begin{enumerate}
\item[(i)]the polar algebras constructed from Clifford systems.
\item[(ii)] the Killing metrized Hsiang algebras.
\end{enumerate}
There are infinitely many nonisomorphic polar algebras (i), occuring in arbitrarily large dimensions, while results in \cite[Ch.~6]{NTVbook} establish that exceptional Hsiang algebras occur in at most finitely many dimensions no greater than $72$. The intersection of (i) and (ii) is small, consisting of the four mutant Hsiang algebras.

\subsection{The Peirce decomposition and the hidden Jordan algebra structure}\label{sec:Peirce}
The results recounted next are needed in Theorem \ref{th:hurwitzdefect} that shows a relation between two numerical invariants associated with a Killing metrized Hsiang algebra.

A principal tool in the structure theory of Hsiang algebras (in particular, in the proof of the finiteness of the isomorphism types of exceptional radial eigencubics) is the Peirce decomposition \eqref{Periced} of an idempotent (for more details see \cite[Ch.~6]{NTVbook}). This is like the Peirce decomposition associated with an idempotent in a Jordan algebra, although a bit more complicated.

By \eqref{self}, for an idempotent $c$ in a metrized commutative algebra $(\Alg,\diamond,h)$, the multiplication operator $L(c)$ is self-adjoint, hence $L(c)$ naturally induces the $h$-orthogonal eigenspace decomposition, called the \textit{Peirce decomposition} associated with $c$:
$$
\Alg=\bigoplus_{i=1}^k \Alg_c(\alpha_i),
$$
where $L(c)=\alpha_i1_{\Alg_c(\alpha_i)}$, $1\le i\le k$, and $\alpha_i$ are (necessarily real) eigenvalues of $L(c)$, called the Peirce spectrum of $c$.

For any nonzero idempotent $c$ in a Hsiang algebra $(\Alg,\diamond,h)$, the Peirce spectrum is $\{1,-1,-\frac12,\frac12\}$ and the associated  Peirce decomposition is
\begin{align}\label{Periced}
\Alg=\Alg_c(1)\oplus \Alg_c(-1)\oplus \Alg_c(-\half12)\oplus \Alg_c(\half12).
\end{align}
Moreover, the eigenvalue $1$ has multiplicity one and it follows from the trace-free condition in \eqref{Hsiang1} that 
\begin{align}\label{n1n2}
\dim \Alg=3\dim \Alg_c(-1)+2\dim \Alg_c(-\half12)-1,
\end{align}
where the \emph{Peirce dimensions}
\begin{align}\label{peircedimensions}
n_1=\dim\Alg_c(-1),\qquad n_2=\dim\Alg_c(-\half12).
\end{align}
and
$$\dim \Alg_c(\half12)=2\dim \Alg_c(-1)+\dim \Alg_c(-\half12)-2 = 2n_{1} + n_{2} - 2
$$
do not depend on the particular choice of $c$. The Peirce dimensions $(n_1,n_2)$ for exceptional eigencubics are shown in Table~\ref{tabs}. For a polar algebra $\Alg=\alg_0\oplus \alg_1$, the Peirce dimensions are given by
$$
(n_1,n_2)=(\dim \alg_0-1,\,\half12\dim\alg_1-\dim \alg_0+2).
$$

One can prove that the subspace $\mathbb{B}_c:=\Alg_c(1)\oplus \Alg_c(-\frac12)$ is  a subalgebra $(\Alg,\diamond,h)$, and moreover, with the modified multiplication
\begin{align}\label{bullet1}
x\ast y=\frac{1}{2}x\diamond y+h(x,c)y+h(y,c)x-2h(x\diamond y,c)c, \qquad x,y\in \mathbb{B}_c,
\end{align}
$\mathbb{B}_c$ becomes a rank 3 Jordan algebra. A remarkable property is that this \textit{Jordan algebra $(\mathbb{B}_c,\ast,h)$ is simple if and only if the original Hsiang algebra $(\Alg,\diamond,h)$ is exceptional}. In the case when $n_2=\dim \Alg_c(-\half12)=0$, the subalgebra $\mathbb{B}_c=\Alg_c(1)\cong \R{}$  is trivial. This corresponds exactly to the five the isoparametric eigencubics. By the Jordan-von Neumann-Wigner classification \cite{JordanNeumann}, the simple formally real rank three Jordan algebras are exactly the $(3d+3)$-dimensional Jordan algebras of Hermitian $3\times 3$-matrices over Hurwitz algebras $\Hurw_d$, $d=1,2,4,8$. This implies that when the Peirce dimension $n_2=\dim \Alg_c(-\half12)$ is positive there holds
\begin{align}\label{d}
n_2=\dim \Alg_c(-\half12)=\dim \mathbb{B}_c-\dim \Alg_c(1)=3d+2,
\end{align}
which combined with \eqref{n1n2} implies that the dimension of $ \Alg$ \emph{is divisible by 3}:
\begin{align}
\label{three}
\dim \Alg=3(2d+1+n_1).
\end{align}
This observation was what originally suggested a relation between exceptional Hsiang algebras and the tripling construction defined in section \ref{sec:trip}.

In this notation, a Hsiang algebra $(\Alg,\diamond,h)$ is mutant if and only if $n_2=2$, which formally corresponds to the case $d=0$ in \eqref{d}.
This motivates the following definition.

\begin{definition}\label{def:hurdim}
Given a Killing metrized algebra, the number $d=d(\Alg)$ such that \eqref{d} holds is called its \textit{Hurwitz dimension}.
\end{definition}

\begin{corollary}
If $(\Alg,\diamond,h)$ is a Killing metrized Hsiang algebra distinct from a Cartan  the relation \eqref{d} holds with $d\in \{0,1,2,4,8\}$.
\end{corollary}

\begin{remark}
The Hsiang algebras corresponding with the isoparametric eigencubics appear in the context of pseudocomposition algebras \cite{Meyberg1}. On the other hand, the Cartan eigencubics were used in \cite{NTV}, \cite{NV2013five} in the construction of non-classical and singular examples of truly viscosity  solutions to fully nonlinear Hessian type uniformly elliptic equations; see also \cite[sec.~3]{Tk19a} for a  friendly and detailed exposition of this construction. In \cite{Tk14}, a one-to-one correspondence between cubic solutions of the general eikonal equation similar to \eqref{Cartan2} and rank three Jordan algebras has been established.
\end{remark}

\begin{remark}
By \cite[Theorem 6.11.1]{NTVbook}, a radial Hsiang algebra satisfies $n_{2} = 0$ if and only if it contains no nonzero $2$-nilpotent element. Via \eqref{grad} it follows that the singular locus of a radial eigencubic is trivial if and only if the radial eigencubic is one of the five radial eigencubics of isoparametric type.
\end{remark}

\section{Quasicomposition algebras}
\label{sec:Hurw}
The quasicomposition algebras defined in this section generalize Hurwitz algebras and cross product algebras.
Hurwitz algebras were defined in Example \ref{ex:tri}. Some basic facts about Hurwitz algebras needed here are recalled now. Their proofs can be found in \cite{Elduque2020, SpringerVeldkamp}.
Let $e$ be the unit in the Hurwitz algebra $(\Alg,\diamond)$. The polar bilinear form
\begin{align}\label{compos2}
n(x,y)=n(x+y)-n(x)-n(y)
\end{align}
induces the trace form $t(x)=n(x,e)\in \Alg^*$ such that any element $x\in\Alg$ satisfies  the quadratic relation
\begin{align}\label{quadrat}
x^2 -t(x)x+n(x)e=0,
\end{align}
and the norm recovered from $t$ by $n(x):=t(x)^2-t(x^2)$ satisfies the composition law \eqref{compos}. It follows from \eqref{compos} that $n(e)=1$ and from \eqref{compos2} that $t(e)=2$.
Then
$$
x^\sigma=t(x)e-x
$$
is an involution on $\Alg$ and the trace bilinear form $h(x,y)=t(x\diamond y)$ is involutively invariant, in other words, $(\Alg, \diamond,\si, h)$ is a metrized algebra with involution.

Any Hurwitz algebra satisfies the \textit{composition identity} (see Lemma~1.3.3 in \cite{SpringerVeldkamp})
\begin{align}\label{compos3}
L(x^\sigma)L(x)y=x^\sigma\diamond (x\diamond y)=n(x)y.
\end{align}

\begin{definition}
A metrized algebra $(\Alg,\diamond, \sigma, h)$ is a \textit{quasicomposition algebra} if
\begin{align}\label{compos4}
L(x)L(x^\sigma)L(x)y=x\diamond (x^\sigma\diamond (x\diamond y)=h(x,x)(x\diamond y)
\end{align}
holds for any $x,y\in \Alg$, where $n(x)=h(x,x)$. A quasicomposition algebra over the real field with a positive definite form $h$ is said to be Euclidean.
\end{definition}

Note that \eqref{compos3} immediately implies \eqref{compos4}.

\begin{proposition}
A unital composition algebra is a quasicomposition algebra.
\end{proposition}

However, a quasicomposition algebra need not be unital. Examples are given below.

An important invariant of a quasicomposition algebra is its deviation from the composition property. By \eqref{compos3}, for any composition (Hurwitz) algebra
\begin{align}\label{traceC}
\trace L(x)L(x^\sigma)=\dim \Alg\cdot  h(x,x).
\end{align}
Such a trace identity holds true for any quasicomposition algebra with a certain integer factor in place of $\dim \Alg$.

\begin{theorem}\label{the:defect}
Suppose $(\Alg,\diamond, \sigma, h)$ is a nonzero  quasicomposition algebra. There is an integer $0\le \delta(\Alg)\le \dim\Alg-1$ such that
\begin{align}\label{traceQ}
\trace L(x)L(x^\sigma)=(\dim \Alg-\delta(\Alg))h(x,x)
\end{align}
for all $x\in \Alg$. Furthermore, $\delta=0$ if and only if $(\Alg,\diamond, \sigma, h)$ is a composition algebra.
\end{theorem}

\begin{proof}
Let $x\ne0$.The defining identity \eqref{compos4} implies the  operator identity
\begin{align}\label{division}
L(x^\sigma)L(x)L(x^\sigma)L(x)=h(x,x)L(x^\sigma)L(x),
\end{align}
which can be rewritten as $P(x)^2=P(x)$, where $P(x)=\frac{1}{h(x,x)}L(x^\sigma)L(x)$. It follows from \eqref{Qass} that
$$
h(x,x)h(P(x)y,z)=h(L(x^\sigma)L(x)y,z)=h(y,L(x^\sigma)L(x)z)=h(x,x)h(y,P(x)z),
$$
so that $P(x)$ is self-adjoint with respect to $h$. This implies that $P(x)$ is an orthogonal projection in $\Alg$, thus
$\trace P(x)=\dim \im P(x)=:m(x)\in \mathbb{Z}$. Since $P(x)$ is not identically zero, this implies $m(x)\in \{1,2,\ldots, \dim \Alg\}$. Moreover, $\trace L(x^\sigma)L(x)-m(x)h(x,x)=0$ identically on $\Alg$. A standard continuity (in the Zariski topology) argument implies that $m(x)$ must be constant on $\Alg$, so that $m(x)=m$. Then $\delta:=\dim \Alg-m$ satisfies \eqref{traceQ}.

By \eqref{traceC}, $\delta(\Alg)=0$ for any Hurwitz algebra. In the converse direction, if $\delta(\Alg)=0$ then the above argument shows that $\dim \im P(x)=\dim \Alg$, therefore $\im P(x)=\im L(x^\sigma)L(x)=\Alg$, hence $P(x)=1$ on $\Alg$, in other words, $L(x^\sigma)L(x)=h(x,x)1_{\Alg}$, implying \eqref{compos3}, so that $(\Alg,\diamond, \sigma, h)$ is a composition algebra.
\end{proof}

The proof of Theorem~\ref{the:defect} implies the following characterization of $\delta(\Alg)$.

\begin{corollary}\label{cor:delta}
Suppose $(\Alg,\diamond, \sigma, h)$ is a nonzero  quasicomposition algebra. For any nonzero $x\in \Alg$
$$
\delta(\Alg)=\dim \ker L(x^\sigma)L(x)=\dim \ker L(x)L(x^\sigma).
$$
\end{corollary}

\begin{definition}
The integer $\delta(\Alg)$ defined by Theorem \ref{the:defect} is the \textit{defect} of the quasicomposition algebra $\Alg$.
\end{definition}

Hurwitz algebras are quasicomposition algebras having defect $0$. The defect measures the deviation of a quasicomposition algebra from the composition property.
There follow some examples of quasicomposition algebras that are not Hurwitz algebras.

\begin{example}[$\delta(\Alg)=0$]\label{ex:com}
As in Example \ref{ex:tri}, the \textit{paracomplex} algebra $\hat{\mathbb{C}}$ is the $2$-dimensional real vector space $\mathbb{C}$ equipped with the commutative multiplication
$x\diamond y=\overline{x}\overline{y},$ where the juxtaposition $xy$ is the standard multiplication in $\mathbb{C}$. Take $\si$ to be the identity on $\hat{\mathbb{C}}$. Then $\si$ is an involution on $\hat{\mathbb{C}}$:
$$
(x\diamond y)^\sigma=x\diamond y=y\diamond x=y^\sigma\diamond x^\sigma.
$$
Because $\hat{\com}$ is exact, it is not unital.
Finally,
$$
x\diamond (x^\sigma\diamond (x\diamond y)=x\diamond (x\diamond (x\diamond y) =x\diamond (x\diamond \bar{x}\bar{y}) = x\diamond \bar{x}xy = n(x)x\diamond y,
$$
implying the quasicomposition identity \eqref{compos4} with $\delta(\Alg)=0$.

Note that the paraquaternions are not a quasicomposition algebra. The preceding proof uses the commutativity of the usual complex multiplication in an essential way.
\end{example}

\begin{example}[$\delta(\Alg)=1$]
A \textit{cross product algebra} is a metrized anticommutative algebra $(\mathbb{A}, \times, \sigma, h)$ equipped with the standard involuiont $x^{\si} := -x$, such that
\begin{align}\label{crossproduct}
x\times (x\times y)=-h(x,x)y+h(x,y)x,
\end{align}
for all $x,y\in \Alg$. Nonzero cross product algebras exist only in dimensions $3$ and $7$ \cite{brown1967vector}. Any simple $3$-dimensional Lie algebra is a cross-product algebra. The case of $\so(3)$ yields the usual cross product on $\rea^{3}$ and a $3$-dimensional real cross product algebra with positive definite $h$ is isomorphic to this example. The imaginary octonions with the cross product equal to the commutator of the usual octonion product give an example of a $7$-dimensional cross-product algebra.

Since the multiplication is anticommutative, the identity \eqref{crossproduct} becomse $x^\sigma\times (x\times y)=h(x,x)y-h(x,y)x$ which implies
$$
x\times (x^\sigma\times (x\times y))=h(x,x)x\times y,
$$
so that $(\mathbb{A}, \times, \sigma, h)$ is a {quasicomposition} algebra. Since $x\times x=0$, by Corollary~\ref{cor:delta} one has $\delta(\Alg)=\dim \ker L(x^\sigma)L(x)\ge1$. On the other hand, if $x\ne 0$ and $y\in  \ker L(x^\sigma)L(x)$ is non-collinear to $x$ then by (i) we have $-h(x,x)y+h(x,y)x=0$, implying $y=0$, so that $\delta(\Alg)=1$.
\end{example}

\begin{example}[$\delta(\Alg)=2$]
In 1978 Domokos and K\"ovesi-Domokos \cite{Domokos} introduced the color algebra $\mathrm{Col}$ which satisfies the triality rule for quarks and an exact superselection rule between leptons and quarks. It can be defined as a 7-dimensional algebra spanned by a unit $e$ and vectors $u_i, i= \pm1, \pm2, \pm3,$ satisfying
$$
u_{\pm i}\circ u_{\pm j}=\epsilon_{ijk} u_{\mp k}, \qquad u_{\pm i}\circ u_{\mp j}=\delta_{ij}e,
$$
where $\epsilon_{ijk}$ is the totally skew-symmetric tensor with $\epsilon_{123} = 1$. By construction, the algebra $\mathrm{Col}$ is metrized with respect to a natural Euclidean metric $h$. In \cite{EldColor} Elduque studied the `imaginary' 6-dimensional subspace $\Alg=e^\bot$ (the orthogonal complement to the unit) with the product $x\times y=\mathbf{pr}_{\Alg}(x\circ y)$) induced by orthogonal projection of $\mlt$ and showed that $\Alg$ satisfies the quasicomposition identity (although he did not use this terminology). An equivalent explicit description of $\Alg$ in terms of the the standard three-dimensional cross product algebra $(\rea^{3}, \cmlt, \si, h)$ with its standard involution $\si = -\Id$ presents it as $\Alg\cong \R{3}\times \R{3}$ with the anticommutative multiplication
\begin{align}
(x',x'')\diamond (y',y'') &= \left(x' \times y' - x'' \times y'', -x' \times y'' - x'' \times y'\right)
\end{align}
and the product Euclidean metric $h$ on $\R{3}\times \R{3}$.
This algebra appears as such in \cite[Theorem $6.6$]{ElduqueMyung} where it is denoted $C_{0}(1, 1, 1)$ and called the \emph{vector color algebra}. Also in \cite[Theorem $3.1$]{ElduqueMyung} it is observed that $\diamond$ satisfies the quasicomposition identity.

An argument similar to that for the cross-product on $\R{3}$ shows that $\delta(\Alg)=2$.
\end{example}

\section{Tripling construction and Hsiang algebras}\label{sec:trip}
This section shows that the quasicompostion algebras arise naturally as building blocks for exceptional Hsiang algebras. The bridge material between these different contexts is the tripling construction described next.

\subsection{The triple algebra construction}
The classical Cayley-Dickson doubling construction produces from a Hurwitz algebra a new Hurwitz algebra having twice the dimension as the original algebra. The most well known proof of the classification of Hurwitz algebras is based on this process \cite{SpringerVeldkamp}. A. Albert extended the doubling construction to algebras with involution \cite{Albert1942}. The tripling constrution described here has a similar flavor, associating with an algebra with involution a commutative algebra having three times the dimension of the original algebra. Just as the doubling construction can be used to classify Hurwitz algebras, the triple construction is useful for studying Hsiang algebras.

The motivation for the tripling construction, described in detail below, following its definition, is the correspondence between metrized commutative algebras and cubic spaces. Succinctly, the trilinear form of a metrized algebra with involution can be viewed as a cubic form on the threefold direct product of the algebra with itself, and via the correspondence this yields the triple product of Definition \ref{def:triple}.

\begin{definition}\label{def:triple}
Given an algebra with involution $(\mathbb{A},\diamond, \sigma)$, its \emph{triple} is the vector space
\begin{align}\label{Tdecomp}
\trip(\mathbb{A})=\mathbb{A}\times \mathbb{A}\times \mathbb{A}
\end{align}
equipped with the evidently commutative multiplication given by
\begin{align}\label{tripleproduct}
\begin{aligned}
(x_1,x_2,x_3)\ast (y_1,y_2,y_3):=(
x_3^\sigma   y_2^\sigma+y_3^\sigma   x_2^\sigma,
x_1^\sigma   y_3^\sigma+y_1^\sigma   x_3^\sigma,
x_2^\sigma   y_1^\sigma+y_2^\sigma   x_1^\sigma
).
\end{aligned}
\end{align}
(Here juxtaposition indicates the product $\diamond$.)
Alternatively, the $\ast$-multiplication by $\xt=(x_1,x_2,x_3)$ is given by the endomorphism
\begin{align}\label{matr}
L_{\ast}(\xt)=R_{\ast}(\xt)=
\left(
  \begin{array}{ccc}
    0 & \sigma R_\diamond(x_3) & \sigma L_\diamond(x_2) \\
    \sigma  L_\diamond(x_3) & 0 & \sigma R_\diamond(x_1) \\
    \sigma R_\diamond(x_2) & \sigma L_\diamond(x_1) & 0 \\
  \end{array}
\right).
\end{align}
\end{definition}
It follows from \eqref{matr} that the triple algebra is exact,
\begin{align}\label{tracefree}
\trace L_\ast(\xt)=0, \qquad \forall x\in \trip(\mathbb{A}).
\end{align}

Although this paper treats only algebras over the real field, Definition \ref{def:triple} works as stated over general ground fields.

The triple algebra has the following motivation. Consider a metrized commutative algebra $(\alg, \diamond , h)$. By definition the trilinear form $\om(x, y, z) = h(x  y, z)$ is completely symmetric. This trilinear form can be viewed as a cubic form on the threefold product $\mathbb{A}\times \mathbb{A}\times \mathbb{A}$. This threefold product can be endowed with the product metric $H = h\times h \times h$, and this makes $(\mathbb{A}\times \mathbb{A}\times \mathbb{A}, H, \om)$ a cubic space. With respect to the metric $H$, the polarization of the cubic form $\om$ determines on $\mathbb{A}\times \mathbb{A}\times \mathbb{A}$ a commutative multiplication $\ast$, and this product is by definition the triple of the original multiplication. By construction the metric $H$ is invariant with respect to the triple product. The first author found this construction in the context of building new metrized commutative algebras from old and described it in \cite{Fox2009, Fox2015, Fox2021a}.

Now consider a semisimple real Lie algebra $\g$ metrized by its Killing form, $B$. The trilinear form $\om(x, y, z) = B([x, y], z)$ is completely antisymmetric, but still can be viewed as a cubic form on the threefold product of the Lie algebra, and the preceding construction goes through as written to produce a commutative multiplication on $\g \times \g \times \g$. This multiplication was introduced, with independent motivations, by M. Kinyon and A. Sagle in \cite{kinyon2002nahm}, where it was called the \emph{Nahm algebra} of the Lie algebra. Actually nothing about the Jacobi identity is used in the Nahm algebra construction and it makes sense for any metrized anticommutative algebra.

This shows that from a metrized commutative or anticommutative algebra there is a natural construction of a metrized commutative algebra having three times the dimension of the original algebra.

The  tripling construction described here was found by both authors looking for a common generalization of the commutative and anticommutative cases. After the fact it is evident that the tripling construction has much in common with the construction of Albert algebras from Hurwitz algebras, but this will be discussed elsewhere.

\medskip
The common construction realizes the triple algebra of a metrized algebra with involution $(\mathbb{A},\diamond, \sigma, h)$ as the polarization of the twisted cubic form $h^{\si}(x  y, z) = h(x  y, z^{\si})$ on $\mathbb{A}\times \mathbb{A}\times \mathbb{A}$. It is straightforward to check that $h$ is involutively invariant in the sense \eqref{invol}-\eqref{Qass} if and only if the twisted bilinear form $h^{\si}$ defined by $h^{\si}(x, y) = \tfrac{1}{2}(h(x, y^{\si}) + h(x^{\si}, y))$ satisfies
\begin{align}
h^{\si}(x^\sigma,y^\sigma)&=h^{\si}(x,y),\label{tinvol}\\
h^{\si}(x  y,z)&=h^{\si}(x,z  y), \quad \forall x,y,z\in \Alg.\label{tQass}
\end{align}
In this case the trilinear form $h^{\si}(x  y, z)$ is completely symmetric. It determines a cubic form $u(x)$ on $\mathbb{A}\times \mathbb{A}\times \mathbb{A}$ with formally the same expression, $u(x) = h^{\si}(x_{1}  x_{2}, x_{3})$. Polarizing this form with respect to the product metric $H = h\times h \times h$ and using the invariance of $h^{\si}$ yields
\begin{align}\label{Htrilinear}
\begin{aligned}
h^{\si}&(x_{1}  y_{2}, z_{3}) + h^{\si}(y_{1}  z_{2}, x_{3}) + h^{\si}(z_{1}  x_{2}, y_{3})\\
&\qquad\qquad\qquad+ h^{\si}(y_{1}  x_{2}, z_{3})+ h^{\si}(z_{1}  y_{2}, x_{3})+ h^{\si}(x_{1}  z_{2}, y_{3})\\
& = h(z_{1}, x_3^\sigma   y_2^\sigma+y_3^\sigma   x_2^\sigma) + h(z_{2}, x_1^\sigma   y_3^\sigma+y_1^\sigma   x_3^\sigma) + h(z_{3},x_2^\sigma   y_1^\sigma+y_2^\sigma   x_1^\sigma )\\
& = H(x\ast y, z),
\end{aligned}
\end{align}
where $\ast$ is the triple product defined in \eqref{tripleproduct}. In short, the cubic form of the triple product is the twisted trilinear form of the original algebra.

\begin{proposition}\label{pro:H}
If $(\mathbb{A},\diamond, \sigma)$ carries an involutively invariant metric $h$ then the commutative algebra $(\trip(\mathbb{A}),\ast)$ is metrized by $H$ defined by $H(x,y):=\sum_{i=1}^3h(x_i,y_i)$. Moreover, the cubic form of $(\trip(\mathbb{A}),\ast, H)$ is the trilinear form of $(\mathbb{A},\diamond, \sigma, h^{\si})$.
\end{proposition}

\begin{proof}
Because $h$ is involutively invariant, the twisted bilinear form $h^{\si}$ is invariant, so \eqref{Htrilinear} shows that the trilinear form $H(x\ast y, z)$ is completely symmetric, which shows the invariance of $H$ with respect to $\ast$. Taking $x = y = z$ in \eqref{Htrilinear} yields $H(x\ast x, x) = 6h^{\si}(x_{1}  x_{2}, x_{3})$, which is the final claim.
\end{proof}


The triple algebra $\trip(\mathbb{A})$ decomposes naturally into an $H$-orthogonal direct sum of vector subspaces
\begin{align}\label{TTTdec}
\trip(\mathbb{A})=\mathbb{T}_1\oplus \mathbb{T}_2\oplus \mathbb{T}_3,
\end{align}
where $\mathbb{T}_1=\mathbb{A}\times 0\times 0$, $\mathbb{T}_2=0\times \mathbb{A}\times 0$, and $\mathbb{T}_3=0\times 0\times \mathbb{A}$ satisfy
\begin{align}\label{TTT}
&\mathbb{T}_i\ast \mathbb{T}_i=\{0\},&&
\mathbb{T}_i\ast \mathbb{T}_j=\mathbb{T}_k, \quad \{i,j,k\}=\{1,2,3\}.
\end{align}
If $\xt=\vt_1+\vt_2+\vt_3\in \trip(\mathbb{A})$, where
\begin{equation}\label{notation}
\vt_1=(x_1,0,0)\in \mathbb{T}_1, \quad \vt_2=(0,x_2,0)\in \mathbb{T}_2, \quad
\vt_3=(0,0,x_3)\in \mathbb{T}_3,
\end{equation}
then
\begin{align}\label{vvv}
\begin{split}
\vt_1\ast \vt_1&=\vt_2\ast \vt_2=\vt_3\ast \vt_3=0, \\
\vt_2\ast \vt_3&=(x_3^\sigma  x_2^\sigma,0,0)\in \mathbb{T}_{1},\\
\vt_3\ast \vt_1&=(0,x_1^\sigma  x_3^\sigma,0)\in \mathbb{T}_{2},\\
\vt_1\ast \vt_2&=(0,0,x_2^\sigma   x_1^\sigma)\in \mathbb{T}_{3}.
\end{split}
\end{align}
implying
$$
\xt\ast \xt=2(x_3^\sigma  x_2^\sigma,x_1^\sigma  x_3^\sigma,x_2^\sigma   x_1^\sigma).
$$
Moreover,
\begin{align}\label{lbull}
L_\ast(\vt_1)^2\vt_2=\vt_1\ast (\vt_1\ast \vt_2)=(0,x_1^\sigma (x_2^\sigma  x_1^\sigma )^\sigma,0)=(0,x_1^\sigma  (x_1  x_2),0)\in \mathbb{T}_2
\end{align}
and similarly
\begin{align}\label{vv}
\begin{split}
L_\ast(\vt_1)^3\vt_2&=
(0,0,((x_2^\sigma  x_1^\sigma)  x_1) x_1^\sigma)=
(0,0,(x_1  (x_1^\sigma  (x_1   x_2)))^\sigma)\\
&=(0,0,\sigma L_\diamond(x_1)L_\diamond(x_1^\sigma)L_\diamond(x_1)x_2)\in \mathbb{T}_3.
\end{split}
\end{align}
Similar relations hold for any pair of $\vt_i$ and $\vt_j$ with $i\ne j$. We record some useful identities needed in the proof of Theorem \ref{th:tuda}.
By \eqref{vvv},
\begin{align}\label{vv3}
\xt\ast \xt=2(\vt_2\ast \vt_3+\vt_3\ast \vt_1+\vt_1\ast \vt_2),
\end{align}
and therefore, using the fact that $H$ is involutively invariant (Proposition~\ref{pro:H}) and \eqref{vvv},
\begin{align}\label{HH}
\begin{split}
H(\xt\ast \xt,\xt)&=H(\vt_2\ast \vt_3,\vt_1)+H(\vt_3\ast \vt_1,\vt_2)+2H(\vt_1\ast \vt_2,\vt_3)
\\
&=6H(\vt_1\ast \vt_2,\vt_3).
\end{split}
\end{align}
Similarly,
\begin{align}\label{HH2}
\begin{split}
\xt\ast (\xt\ast\xt)
=&2\underbrace{(\vt_2\ast(\vt_2\ast\vt_1)+\vt_3\ast(\vt_3\ast\vt_1)}_{:=\wt_1}+
\underbrace{\vt_1\ast(\vt_1\ast\vt_2)+\vt_3\ast(\vt_3\ast\vt_2)}_{:=\wt_2}\\
&+\underbrace{\vt_1\ast(\vt_1\ast\vt_3)+\vt_2\ast(\vt_2\ast\vt_3))}_{:=\wt_3}=:2\wt_1+2\wt_2+2\wt_3, \qquad \wt_i\in \mathbb{T}_i.
\end{split}
\end{align}

\begin{theorem}\label{th:tuda}
Let $(\Alg,\diamond,\sigma,h)$ be an algebra with involution equipped with an involutively invariant metric $h$. Its triple $(\trip(\alg), \ast)$ metrized by $H = h\times h \times h$ is a Hsiang algebra if and only if $(\Alg,\diamond,\sigma,h)$ a Euclidean quasicomposition algebra. In this case $(\trip(\alg), \ast)$ is normalized by $\theta =4/3$.
\end{theorem}

\begin{proof}
If $(\Alg,\diamond,\sigma,h)$ is a quasicomposition algebra, then the identity  $L_\diamond(x)L_\diamond(x^\sigma)L_\diamond(x)=h(x,x)L_\diamond(x)$ holds for any $x\in \mathbb{A}$, so \eqref{vv} implies
\begin{align}\label{vv2}
L_\ast(\vt_i)^3\vt_j=H(\vt_i,\vt_i)\vt_i\ast \vt_j.
\end{align}
Combining \eqref{HH2} with \eqref{vv3} and \eqref{TTT} yields
\begin{align}\label{HH3}
\begin{split}
H(\xt\ast \xt,\xt\ast (\xt\ast\xt))&=2H(\vt_2\ast\vt_3, \wt_1)
+2H(\vt_3\ast\vt_1, \wt_2)+2H(\vt_1\ast\vt_2, \wt_3).
\end{split}
\end{align}
By \eqref{vv2} there hold
\begin{align*}
H(\vt_2\ast\vt_3, \wt_1)&=H(\vt_2\ast\vt_3, \vt_2\ast(\vt_2\ast\vt_1))
+H(\vt_2\ast\vt_3,\vt_3\ast(\vt_3\ast\vt_1)\\
&=H(\vt_3, \vt_2\ast(\vt_2\ast(\vt_2\ast\vt_1)))
+H(\vt_2\ast\vt_3,\vt_3\ast(\vt_3\ast(\vt_3\ast\vt_1))\\
&=H(\vt_2,\vt_2)\, H(\vt_3, \vt_2\ast\vt_1)
+H(\vt_3,\vt_3)\, H(\vt_2, \vt_3\ast\vt_1)\\
&=[H(\vt_2,\vt_2)+H(\vt_3,\vt_3)]\, H(\vt_1, \vt_2\ast\vt_3),
\end{align*}
and the similar identities obtained by permuting the indices.
Substituting this in \eqref{HH3} and using \eqref{HH} yields
\begin{align}\label{HH4}
\begin{split}
H(\xt\ast \xt,\xt\ast (\xt\ast\xt))&
=8[H(\vt_1,\vt_1)+H(\vt_2,\vt_2)+H(\vt_3,\vt_3)]\, H(\vt_1, \vt_2\ast\vt_3)\\
&=\frac43 H(\xt,\xt)\,H(\xt\ast \xt,\xt).
\end{split}
\end{align}
Because $\ast$ is exact, this last relation shows that $\trip(\mathbb{A})$ satisfies the Hsiang equation \eqref{Hsiang1} for $\theta =4/3$.

In the converse direction, suppose that $\trip(\mathbb{A})$ satisfies the Hsiang equation \eqref{Hsiang1} with $\theta =4/3$.
Suppose that $\xt=\vt_1+\vt_2$, $\vt_i\in \mathbb{T}_i$. By \eqref{vvv}
$$
\xt\ast \xt=2\vt_1\ast \vt_2=(0,0,2x_2^\sigma   x_1^\sigma)\in \mathbb{T}_3,
$$
hence $H(\xt\ast \xt, \xt)=0$ and
$$
(\xt\ast \xt)\ast(\xt\ast \xt)=0.
$$
Also, using \eqref{lbull},
$$
(\xt\ast \xt)\ast \xt=2\vt_1\ast \vt_2\ast (\vt_1+\vt_2)=
(2(x_1  x_2)  x_2^\sigma,2x_1^\sigma  (x_1  x_2),0)\in \mathbb{T}_1\oplus \mathbb{T}_2
$$
Applying \eqref{Hsiang2} for $\theta =4/3$ to $\xt=\vt_1+\vt_2$ yields
$$
4((\xt\ast \xt)\ast \xt)\ast \xt-4 H(\xt,\xt) \xt\ast \xt=0.
$$
Using
$$
((\xt\ast \xt)\ast \xt)\ast \xt=
2(0,\,0,\, x_2^\sigma (x_2 (x_2^\sigma   x_1^\sigma))+(x_2^\sigma   x_1^\sigma)   x_1 )   x_1^\sigma  )
$$
and $H(\xt,\xt)=h(x_1,x_1)+h(x_2,x_2)$ yields
$$
x_2^\sigma (x_2 (x_2^\sigma   x_1^\sigma))+(x_2^\sigma   x_1^\sigma)   x_1 )   x_1^\sigma=
(h(x_1,x_1)+h(x_2,x_2))x_2^\sigma   x_1^\sigma.
$$
Setting $x_1\to tx_1$, and equating terms of equal degree in $t$ yields htat
$$
x_2^\sigma (x_2 (x_2^\sigma   x_1^\sigma))=h(x_2,x_2)x_2^\sigma   x_1^\sigma,\qquad
(x_2^\sigma   x_1^\sigma)   x_1 )   x_1^\sigma=h(x_1,x_1)x_2^\sigma   x_1^\sigma
$$
holds for any $x_1,x_2\in \Alg$, implying \eqref{compos4}. The theorem is proved.
\end{proof}

We give some immediate corollaries of Theorem \ref{th:tuda}.

\begin{corollary}
Let $(\Alg,\diamond,\sigma,h)$ be a Euclidean quasicomposition algebra. The cubic form
\begin{align}\label{cubicf}
u(\xt)=\frac16H(x\ast x,x)\in S^{3}\alg^{\star},
\end{align}
of the triple algebra $(\trip(\alg), \ast)$ metrized by $H = h\times h \times h$ is a radial eigencubic satisfying \eqref{Hsiang1} with $\theta =4/3$.
\end{corollary}

\begin{example}\label{ex:trial}
The preceding implies that the triple of a composition algebra (in particular, of a Hurwitz algebra) is a Hsiang algebra and the corresponding cubic form \eqref{cubicf} is a radial eigencubic satisfying \eqref{Hsiang1} with $\theta =4/3$. This provides a short conceptual proof that the Hurwitz eigencubics of Example~\ref{ex:tri} above are in fact radial eigencubics.
\end{example}

By \eqref{TTT}, a triple algebra contains $2$-nilpotents, so the singular locus of the radial eigencubic determined by the triple of a quasicomposition algebra is always nontrivial.

\begin{corollary}\label{cor:polar}
The triple algebra $(\trip(\alg), \ast)$ of a Hurwitz algebra $\alg$ is a mutant polar algebra.
\end{corollary}

\begin{proof}
In the notation of \eqref{TTTdec}, let $\mathbb{B}_0:=\mathbb{T}_1$ and $\mathbb{B}_1:=\mathbb{T}_2\oplus \mathbb{T}_3,$ such that $\trip(\alg)=\mathbb{B}_0\oplus\mathbb{B}_1$. By \eqref{TTT},
$\mathbb{B}_0\ast \mathbb{B}_0=0$ and $\mathbb{B}_1\ast \mathbb{B}_1\subset\mathbb{B}_0$, which imply (i) and (ii) in Definition~\ref{def:Clifford}.
Next, using the notation \eqref{notation}, for an arbitrary $v_i\in \mathbb{T}_1$ and $v_2\in \mathbb{T}_2$ it follows from \eqref{lbull} and \eqref{compos3} that
$$
L_\ast(\vt_1)^2\vt_2=(0,x_1^\sigma  (x_1  x_2),0)=h(x_1,x_1)\,(0,0,x_2)=H(v_1,v_1)v_2\in \mathbb{T}_2,
$$
which implies (iii) in Definition~\ref{def:Clifford}. Because $\trace L_\ast(x)=0$, there holds (iv) in Definition~\ref{def:Clifford}, and the desired conclusion follows.
\end{proof}


\begin{remark}
There are \emph{weak} versions of quasicomposition algebras in which the symmetric bilinear form $h$ is allowed to be degenerate and interesting examples can be found. Via the tripling construction such degenerate quasicomposition algebras yield degenerate Hsiang algebras in the sense of Remark \ref{rem:weak}.
An example is given by the $6$-dimensional real vector space $\alg = \gl(2, \rea)\times \rea^{2}$ with the product $(X, u)  (Y, v) = (XY, \bar{X}v + Yu)$ where $\bar{X} = \tr(X)\id_{\rea^{2}} - X$,  the involution $(X, u)^{\si} = (\bar{X}, -u)$, and the metric $h((X, u), (Y, v)) = \tfrac{1}{2}\tr(\bar{X}Y + \bar{Y}X)$ given by polarizing the quadratic form $q(X, u) = \det X= \tfrac{1}{2}\tr(\bar{X}X)$. Because $\bar{X}X = \det(X)\id_{\rea^{2}}$, this solves $L_{ }((X, u)^{\si})L_{ }(X, u) = \det(X)\id_{\alg} = h((X, u), (X, u) \Id_{\alg}$, so solves \eqref{compos4}, though $h$ has rank $4$ rather than $6$. Such weak quasicomposition algebras are not discussed further here.
\end{remark}

\section{Realizability of exceptional Hsiang algebras as triples}\label{sec:Delta}

The triple $\trip(\alg)$ of a quasicomposition algebra $(\Alg,\diamond,\sigma,h)$ is a Hsiang algebra. However, because $\dim \trip(\Alg)=3\dim \Alg$, the dimension of a triple algebra is divisible by $3$, so not every Hsiang algebra can be realized as the triple of a quasicomposition algebra. On the other hand, by \eqref{three} the dimension of an exceptional eigencubic that is not a Cartan isoparametric eigencubic is also divisible by 3. This suggests characterizing those Hsiang algebras \textit{realizable} as triple algebras.

\begin{theorem}\label{th:hurwitzdefect}
Let $(\Alg,\diamond,\sigma,h)$ be a Euclidean quasicomposition algebra with defect $\delta(\Alg)$ defined by \eqref{traceQ}. Then the triple $(\trip(\Alg),\ast,H)$ is a Killing-metrized Hsiang algebra (so is exceptional or mutant). Furthermore,
\begin{align}\label{did}
\delta(\Alg)=d(\trip(\Alg)),
\end{align}
where $d(\trip(\Alg))$ is the Hurwitz dimension of $\trip(\Alg)$ (see Definition \ref{def:hurdim}).
\end{theorem}

\begin{proof}
Tracing the square of \eqref{matr} and using the properties of the trace, the identities \eqref{RL}, and \eqref{traceQ} of Theorem \ref{the:defect} yields
\begin{align*}
\trace L_\ast(\xt)^2&=\trace \left(
  \begin{array}{ccc}
    0 & \sigma R_\diamond(x_3) & \sigma L_\diamond(x_2) \\
    \sigma  L_\diamond(x_3) & 0 & \sigma R_\diamond(x_1) \\
    \sigma R_\diamond(x_2) & \sigma L_\diamond(x_1) & 0 \\
  \end{array}
\right)^2\\
&=2\trace (\sigma R_\diamond(x_3)\sigma L_\diamond(x_3)+\sigma L_\diamond(x_2)\sigma R_\diamond(x_2)+\sigma L_\diamond(x_1)\sigma R_\diamond(x_1))\\
&=2\trace (L_\diamond(x_1^\sigma) L_\diamond(x_1)+L_\diamond(x_2^\sigma) L_\diamond(x_2)+L_\diamond(x_3^\sigma) L_\diamond(x_3))\\
&=2(\dim \Alg-\delta(\Alg))\sum_{i=1}^3h(x_i,x_i)\\
&=2(\dim \Alg-\delta(\Alg))H(\xt,\xt).
\end{align*}
Thus, the Killing trace-form is proportional to the invariant metric $H$, hence $(\trip(\Alg),\ast,H)$ is a Killing-metrized Hsiang algebra.

Next, let $c$ be a nonzero idempotent in $\mathbb{T}:=(\trip(\Alg),\ast,H)$. Then, by \eqref{length},
\begin{align}\label{LT1}
\trace L_\ast(c)^2=2(\dim \Alg-\delta(\Alg))H(c,c)=\frac32(\dim \Alg-\delta(\Alg)).
\end{align}
On the other hand, combining \eqref{Periced} with \eqref{d}, yields directly
\begin{equation}\label{LT2}
\begin{split}
\trace L_\ast(c)^2&=1+\dim \mathbb{T}_c(-1)+\frac14(\dim \mathbb{T}_c(-\frac12)+\dim \mathbb{T}_c(\frac12))\\
&=\frac32(\dim \mathbb{T}_c(-1)+d+1).
\end{split}
\end{equation}
By \eqref{three}, $3\dim \Alg=\dim \mathbb{T}=3(2d+1+\dim \mathbb{T}_c(-1))$, hence
$$
\dim \Alg=2d+1+\dim \mathbb{T}_c(-1).
$$
Equating \eqref{LT1} with \eqref{LT2} yields
$$
\dim \mathbb{T}_c(-1)+d+1=\dim \Alg-\delta(\Alg)=2d+1+\dim \mathbb{T}_c(-1)-\delta(\Alg),
$$
which implies \eqref{did}.
\end{proof}

By Theorem \ref{th:hurwitzdefect} the defect of a quasicomposition algebra equals the Hurwitz dimension of its triple algebra which is a Killing metrized Hsiang algebra, hence we have

\begin{corollary}
The defect of a quasicomposition algebra  $\delta(\Alg)\in \{0,1,2,4,8\}$.
\end{corollary}

It should be possible to control the defect of a quasicomposition algebra directly by associating a Hurwitz algebra with the operator $L(x)L(x^{\si})$.

\begin{corollary}
A quasicomposition algebra has dimension at most $24$.
\end{corollary}
\begin{proof}
The triple algebra of a quasicomposition algebra is a Killing metrized Hsiang algebra, so this follows by inspection of Table~\ref{tabs}.
\end{proof}

\section{Concluding remarks}
The definition of Hsiang algebras is motivated by a problem in differential geometry and this context leads to a focus on solutions over the real field metrized by a Euclidean metric. However the problem of classifying Hsiang algebras makes sense for real algebras with metrics of other signatures and over arbitrary fields. The differential geometric context motivates an initial focus on Euclidean metrics. Relaxing this restriction to include isotropic metrics is necessary for considering algebraically closed fields but is also relevant for the problem of finding algebraic mean curvature zero hypersurfaces in pseudo-Riemannian space forms.

%

Comparision with division algebras, composition algebras, and the like suggests that it is useful initially to focus on base fields that are algebraically closed or real so as to avoid issues that essentially arithmetic (rather than related to the structural properties of the class of algebras under consideration). For example, the classical Hurwitz theorem says that over the real field there are four unital composition algebras, while over the rationals there infinite families in each dimension. On the other hand, that there are only four dimensions in which there occur unital composition algebras is true independently of the base field. This is a genuinely structural statement that does not depend on arithmetic. The expectation is that the general structural features described here and summarized in Diagram \ref{diagram} remain true over general fields.

The notion of quasicomposition algebra arises because of Theorem \ref{th:tuda} - these are exactly the algebras with involution whose triples are Hsiang. However, the notion appears natural in and of itself when it is seen to encompass a slew of closely related algebras, namely unital composition algebras, cross product algebras, the six-dimensional vector color algebras, and twisted forms of all of the preceding, all of which are constructed from Hurwitz algebras. Examples are known in dimensions $0$ (empty algebra), $1$ (the field itself), $2$, $3$, $4$, $6$, $7$, and $8$. In dimension $5$ there are \emph{weak} quasicomposition algebras for which the invariant bilinear form has rank $3$, but it seems likely that there does not exist a $5$-dimensional quasicomposition algebra. A complete classification of quasicomposition algebras is work in progress.

However, there do exist Hsiang algebras in dimensions 
$15$, $27$, $30$, $54$, etc.
This suggests that either there are corresponding quasicomposition algebras in dimensions $5$, $9$, $10$, $18$, etc. or that the definition of quasicomposition algebra needs to be modified or relaxed in some minor way.

A fundamental feature of quasicomposition algebras is the constant rank property, that the left and right multiplication endomorphisms have constant rank, equal to the dimension minus the defect. By itself such a constant rank condition defines a larger (probably much larger) class of algebras. Such a condition admits many interesting variations. It would be interesting to study such conditions in detail.

There is a striking resemblance between radial eigencubics (and, in particular, the above dichotomy), and isoparametric hypersurfaces in the spheres.
Namely, there are  infinitely many isoparametric hypersurfaces (all having four distinct principal curvatures), the FKM-type isoparametric hypersurfaces \cite{FKM}, naturally constructed from symmetric Clifford systems in a manner very similar to the radial eigencubics \eqref{uA}. On the other hand, there are only finitely many non-FKM-type isoparametric hypersurfaces \cite{CecilAnnals}. There are reasons to think the similarity seems is not accidental and this suggests that there should exist a correspondence relating radial eigencubics and isoparametric hypersurfaces in spheres.

\bibliographystyle{plain}

\def\cprime{$'$}

\end{document}